\documentclass[10pt]{amsart}
\usepackage{mathrsfs}
\usepackage{amsfonts}
\usepackage{amssymb}
\usepackage{amsfonts, amscd, amsmath, mathrsfs, amssymb, amsthm, amsxtra, stmaryrd, bbding, epsfig, graphicx, latexsym, url,xcolor}
\usepackage[papersize={7.5in,10.5in},textwidth=14.9cm,textheight=21cm,centering]{geometry}

\usepackage[textsize=small,color=yellow]{todonotes}

\usepackage[normalem]{ulem}
\newcommand\redsout{\bgroup\markoverwith{\textcolor{red}{\rule[0.5ex]{2pt}{0.4pt}}}\ULon}
\usepackage{soul}
\setstcolor{red}
\usepackage{cancel}

\usepackage{comment}
\usepackage{enumerate,enumitem}

\usepackage{xcolor}
\definecolor{cite}{rgb}{0.00,0.60,1.00}
\definecolor{url}{rgb}{1.00,0.10,0.80}
\definecolor{link}{rgb}{0.00,0.00,1.00}
\usepackage[colorlinks,linkcolor=link,urlcolor=url,citecolor=cite,pagebackref,breaklinks]{hyperref}

\DeclareFontFamily{U}{mathx}{\hyphenchar\font45}
\DeclareFontShape{U}{mathx}{m}{n}{
      <5> <6> <7> <8> <9> <10>
      <10.95> <12> <14.4> <17.28> <20.74> <24.88>
      mathx10
      }{}
\DeclareSymbolFont{mathx}{U}{mathx}{m}{n}
\DeclareMathAccent{\widecheck}{\mathalpha}{mathx}{"71}
\allowdisplaybreaks

 \usepackage{caption} 
\numberwithin{equation}{section}

\allowdisplaybreaks

\newtheorem*{theorem*}{Theorem}
\newtheorem{theorem}{Theorem}[section]
\newtheorem{lemma}{Lemma}[section]

\newtheorem{proposition}{Proposition}[section]

\newtheorem{corollary}{Corollary}[section]
\newtheorem{conjecture}{Conjecture}[section]
\newtheorem{remark}{Remark}[section]

\makeatletter
\newcounter{roem}
\renewcommand{\theroem}{\Roman{roem}}

\newcommand{\c@org@eq}{}
\let\c@org@eq\c@equation
\newcommand{\org@theeq}{}
\let\org@theeq\theequation

\newcommand{\setroem}{
\let\c@equation\c@roem
 \let\theequation\theroem}

\newcommand{\setarab}{
\let\c@equation\c@org@eq
\let\theequation\org@theeq}
\makeatother

\newcommand{\ud}{\mathrm{d}}
\newcommand{\ue}{\mathrm{e}}

\newcommand{\Sym}{\mathrm{Sym}}
\newcommand{\kl}{\mathrm{Kl}}
\newcommand{\Frob}{\mathrm{Frob}}

\newcommand{\Tr}{\mathrm{Tr}}

\newcommand{\SL}{\mathrm{SL}}
\newcommand{\GL}{\mathrm{GL}}
\newcommand{\Gal}{\mathrm{Gal}}
\newcommand{\supp}{\mathrm{supp}}

\DeclareMathOperator{\Mod}{mod}

\renewcommand{\bmod}[1]{\,(\Mod{ #1})}

\newcommand{\bk}{\mathbf{k}}

\newcommand{\bF}{\mathbf{F}}

\newcommand{\bQ}{\mathbf{Q}}

\newcommand{\bZ}{\mathbf{Z}}

\newcommand{\btheta}{\boldsymbol\theta}

\newcommand{\cM}{\mathcal{M}}

\newcommand{\cP}{\mathcal{P}}

\newcommand{\fp}{\mathfrak{p}}

\newcommand{\fD}{\mathfrak{D}}

\newcommand{\fN}{\mathfrak{N}}

\newcommand{\blue}[1]{{\color{blue}#1}}

\definecolor{brown}{RGB}{165,42,42} 

\def\le{\leqslant}
\def\leq{\leqslant}
\def\ge{\geqslant}
\def\geq{\geqslant}

\usepackage{graphicx}
\usepackage{tikz}

\begin{document}

\vglue -2mm

\title{Equidistribution of Kloosterman sums over function fields}

\author{Lei Fu}
\address{Yau Mathematical Sciences Center, Tsinghua University}
\email{leifu@tsinghua.edu.cn}

\author{Yuk-Kam Lau}
\address{
Weihai Institute for Interdisciplinary Research, Shandong University, China \mbox{\rm and}
Department of Mathematics, The University of Hong Kong, Pokfulam Road, Hong Kong}
\email{yklau@maths.hku.hk}

\author{Wen-Ching Winnie Li}
\address{Department of Mathematics, Pennsylvania State University,
University Park, PA 16802, USA}
\email{wli@math.psu.edu}

\author{Ping Xi}
\address{School of Mathematics and Statistics, Xi'an Jiaotong University, Xi'an 710049, P. R. China}
\email{ping.xi@xjtu.edu.cn}




\begin{abstract}
We prove the Sato--Tate distribution of Kloosterman sums over function fields with explicit error terms, when the places vary in arithmetic progressions or short intervals. A joint Sato--Tate distribution of finitely many Kloosterman sums is also proved. The arguments in this paper also apply to local systems with $\mathrm{SL}(2)$ monodromy and suitable ramification restrictions.
\end{abstract}
\vglue -3mm
\maketitle

\setcounter{tocdepth}{1}

\section{Introduction and main results}

\subsection{Kloosterman sums over residue rings and finite fields}

The classical Kloosterman sum, attributed to Kloosterman's 1926 paper, is an exponential sum over a residue ring $\bZ/c\bZ$ with $c\in\bZ^+$. It appears on many occasions in modern analytic number theory due to its fruitful and deep applications.
There is another kind of Kloosterman sums, more rooted in algebraic geometry, defined over a finite field $\bF_q$ of characteristic $p$. When $c=q$ is a prime $p$, these two versions of Kloosterman sums coincide and are given by 
\begin{align*}
K(p,a):=\sum_{x\in(\bZ/p\bZ)^\times}\ue\Big(\frac{x+a/x}{p}\Big) 
\end{align*}
where $a\in\bZ$. 

A celebrated bound of Weil \cite{We48} asserts that
$|K(p,a)|\leqslant2\sqrt{p}$
for all primes $p$ and all $a\in\bZ$. Consequently $K(p,a)$, being real-valued, can be expressed as
\begin{align*}
K(p,a)=2\sqrt{p}\cos\theta_p(a)
\end{align*}
for some $\theta_p(a)\in[0,\pi]$.  Katz \cite{Ka88} showed that these angles $\theta_p(a)$ are equidistributed in the following sense.
\begin{theorem}[Katz, \cite{Ka88}]\label{thm:VST} 
The set $\{\theta_p(a):a\in\bF_p^\times\}$ is equidistributed in $[0,\pi]$ with respect to the Sato--Tate measure 
$\mathrm d\mu_{\mathrm{ST}}:=\frac{2}{\pi}\sin^2\theta\ud\theta$ as $p\rightarrow +\infty.$
\end{theorem}

This result is known as the vertical Sato--Tate theorem for Kloosterman sums.  It is natural to ask whether the horizontal Sato--Tate distribution described below also holds. 
\begin{conjecture}[Katz, \cite{Ka80}]\label{conj:HST}
Fix a non-zero integer $a$. The set $\{\theta_p(a):p \nmid a \}$ is equidistributed in $[0,\pi]$ with respect to the Sato--Tate measure 
$\mathrm d\mu_{\mathrm{ST}}.$
\end{conjecture}

Conjecture \ref{conj:HST} is reminiscent of the original Sato--Tate conjecture, which was first formulated independently by Sato
and Tate in the context of elliptic curves, and then reformulated and extended to
the framework of cuspidal Hecke eigenforms for congruence subgroups of $\mathrm{SL}_2(\bZ)$ 
by Serre \cite{Se68}, predicting a similar equidistribution 
of Hecke eigenvalues of such cusp forms.
This conjecture has been confirmed by Barnet-Lamb, Geraghty, Harris and 
Taylor \cite{BGHT11} for non-CM, holomorphic cuspidal newforms of weight $k\geqslant2$.  
The works of Newton and Thorne \cite{NT21I, NT21II} re-established this fact as a consequence of the symmetric power functoriality for holomorphic modular forms. 

As for Conjecture \ref{conj:HST}, one might ask whether the Kloosterman sums $K(p,a)$, up to the sign, 
would appear as Hecke eigenvalues
of a suitable Hecke--Maass eigenform of $\GL_2(\bQ)$ for all large primes $p$ (cf. Katz \cite[p.15]{Ka80}). An affirmative answer to this question 
would allow us to construct $L$-functions as analytic tools, which can be utilized to study the Sato--Tate distribution of Kloosterman sums as predicted by Conjecture \ref{conj:HST}. 
However, the works of Booker \cite{Bo00} and Xi \cite{Xi20} indicate the lack of correlation between Hecke eigenvalues of Maass forms and Kloosterman sums over finite residue rings. As a result, the aforementioned affirmative answer seems rather unlikely.
On the other hand, the work of Chai and Li \cite{CL03} (see Theorem \ref{thm:ChaiLi} below) illustrates a different phenomenon in the situation of function fields.

\subsection{Kloosterman sums over function fields}
Let $K$ be a function
field with the constant field being a finite field $\mathbf F_q$ of characteristic $p$ with $q$ elements. 
Fix a non-trivial additive character of 
$\psi: \bF_q\to \mathbf C^*$. {\it In what follows, $a$ is an element in $K$ but not in $\bF_q$.} 
For any place $\fp$ of $K$, 
let $\bF_\fp$ be the residue field of the completion of $K$ at $\fp$. If $\fp$ is outside the zeros and poles of $a$, 
we denote by $\bar a$ the image of $a$ in $\mathbf F_{\fp}$. Consider the Kloosterman sum
\begin{align}\label{ksff}
\kl(\fp,a):=\sum_{x\in\bF_\fp^\times}\psi(\Tr_{\bF_\fp/\bF_q}(x+\bar a/x)).
\end{align}
As proved by Weil \cite{We48}, we have $|\kl(\fp,a)|\leqslant2\sqrt{N\fp}$, where $N\fp = |\bF_\fp|$ is the norm of the place $\fp$.
Correspondingly, we define $\theta_\fp(a)\in[0,\pi]$ via
\begin{align*}
\kl(\fp,a)&=2\sqrt{N\fp}\cos\theta_\fp(a).
\end{align*} 

\begin{theorem}[Chai--Li, \cite{CL03}]\label{thm:ChaiLi}
~

$(i)$ There exists an automorphic form $f$ of $\mathrm{GL}_2$ over $K$ which is an eigenfunction of the Hecke
operator $T_\fp$ with eigenvalue
$-\kl(\fp,a)$ at all places $\fp$ of $K$ outside the zeros and poles of $a$, that is,
\[L(s,f)\sim\prod_{\fp\text{ good}}\Big(1+\frac{\kl(\fp,a)}{N\fp^s}+\frac{1}{N\fp^{2s-1}}\Big)^{-1}.\]

$(ii)$ 
The set $\{\theta_\fp(a) : \fp~ {\rm outside~ the~ zeros~ and~ poles~ of~} a~{\rm and}~ |\bF_{\fp}| \le M\}$ is equidistributed in $[0,\pi]$ with respect to the Sato--Tate measure $\mathrm{d}\mu_{\mathrm{ST}}=\frac{2}{\pi}\sin^2\theta\ud\theta$ as $M\to +\infty$.
\end{theorem}

The Kloosterman sum $\kl(\fp,a)$ can also be formulated in the framework  of Deligne and Katz. 
For each prime number $\ell$ distinct from $p$, fix an isomorphism of fields $\iota: \overline{\mathbf Q}_\ell\stackrel\cong\to\mathbf C$. 
We use this isomorphism to relate $\ell$-adic numbers and complex numbers. 
Let $X$ be the smooth geometrically connected projective algebraic curve over
$\bF_q$ with the function field $K$. Places in $K$ are in one-to-one correspondence with closed points in $X$, and we 
denote them by the same notation. 
The element $a\in K$ defines a finite dominant morphism $$a:X\to \mathbb P^1$$
to the projective line $\mathbb P^1$ over $\mathbf F_q$. Let 
$\mathbb G_m=\mathbb P^1-\{0,\infty\}$ and let 
$$U=a^{-1}\mathbb G_m=X-\mathrm{supp}(a),$$
where $(a)$ denotes the principal divisor on $X$ defined by $a$. Then $a$ induces a finite dominant 
morphism $$a: U\to \mathbb G_m.$$ In \cite[Sommes trig. 7.8]{De77},
Deligne constructed a smooth $\overline{\mathbf Q}_\ell$-adic sheaf of rank $2$ on $\mathbb G_m$, called the Kloosterman sheaf $\mathrm{Kl}$, 
punctually pure of weight $1$ such that the Frobenius action on the stalk $\mathrm{Kl}_{\bar \lambda}$ of the Kloosterman sheaf at the geometric point $\bar \lambda$ over the $\bF_{q^n}$-point $\lambda$ of $\mathbb G_m$ is equal to
$$\iota\mathrm{Tr}(\mathrm{Frob}_\lambda, \mathrm{Kl}_{\bar \lambda})=-\sum_{x\in\bF_{q^n}^\times}\psi
\big(\Tr_{\bF_{q^n}/\bF_q}(x+\lambda/x)\big).$$ 
For the smooth sheaf $a^*\mathrm{Kl}$ on $U$, we have 
$$\iota\mathrm{Tr}(\mathrm{Frob}_\fp, (a^*\mathrm{Kl})_{\bar \fp})=-\kl(\fp,a).$$ Let $\eta$ be the generic point of $U$. 
Then the stalk $(a^*\mathrm{Kl})_{\bar \eta}$ forms a compatible family of $\ell$-adic representations of dimension $2$ of the absolute Galois group 
$\mathrm{Gal}(\overline K/K)$ of $K$ unramified at places $\fp\not\in\mathrm{supp}(a)$
with trace $-\kl(\fp, a)$ at $\mathrm{Frob}_\fp$. Here $\overline K$ denotes a separable closure of $K$.

Since $a: U\to \mathbb G_m$ is a finite dominant morphism, the homomorphism
$$a_*:\pi_1(U,\bar\eta)\to \pi_1(\mathbb G_m,\bar\eta)$$ induced by $a$ on fundamental groups is injective and $\mathrm{im}\,a_*$ is 
a subgroup of $\pi_1(\mathbb G_m,\bar\eta)$ of finite index. 
The geometric monodromy group (resp. the arithmetic monodromy group) of $\mathrm{Kl}$ (resp. $\mathrm{Kl}(1/2))$
is $\SL_2$ by \cite[11.1]{Ka88} (resp. \cite[11.3]{Ka88}), which implies the equidistribution of $\kl(\fp, a)$ with respect to the Sato-Tate measure, as in Theorem \ref{thm:ChaiLi} (ii).  Recall that $\mathrm{Kl}(1/2)=\mathrm{Kl}\otimes\overline{\mathbf Q}_\ell(1/2)$, 
where for any rational number $r$, 
$\overline{\mathbf Q}_\ell(r)$ is the inverse image on $X$ of the sheaf on $\mathrm{Spec}\,\mathbf F_q$ defined by the character 
$$\mathrm{Gal}(\overline{\mathbf F}_q/\mathbf F_q)\to \overline{\mathbf Q}_\ell^\times, \quad 
\mathrm{Frob}\to 1/q^{r}.$$ We have 
$$\iota\mathrm{Tr}\Big(\mathrm{Frob}_\fp, (a^*\mathrm{Kl}(1/2))_{\bar \fp}\Big)=-\frac{\kl(\fp,a)}{\sqrt{N\fp}}.$$

\subsection{Main results}
While Conjecture \ref{conj:HST} remains beyond reach, we present various kinds of equidistribution results for its counterpart over function fields, and track manipulation in this framework to provide explicit error terms. More precisely, we study the distribution of the Kloosterman sums $\kl(\fp,a)$, with $a$ fixed and $\mathfrak p$ varying in short intervals or arithmetic progressions, or, more generally, with $\Frob_{\mathfrak p}$ in conjugacy classes of a finite Galois group. 
Our main results are described below.

For a function field $K$, denote by $\Pi_d(K)$ the set of places in $K$ of degree $d$.
Let $E/K$ be a finite Galois extension with discriminant $\fD_{E/K}$ and Galois group $\Gal(E/K)$. Then $\mathrm{supp}\,\mathfrak D_{E/K}$
consists of places of $K$ ramified in $E$. Let $C$ be a conjugacy class in $\Gal(E/K)$. Put
\begin{align*}
\Pi_d(E/K,C)=\{\fp\in \Pi_d(K):\fp\not\in\supp\,\fD_{E/K},~\Frob_{\fp}=C\}. 
\end{align*}
The Chebotarev density theorem says that 
$$\lim_{d\to+\infty} \frac{|\Pi_d(E/K, C)|}{|\Pi_d(K)|}=\frac{|C|}{[E:K]}.$$
The following theorem gives an equidistribution of the Kloosterman sums $\kl(\fp,a)$ as $\fp$ varies in $\Pi_d(E/K,C)$ for $d$ large.

\begin{theorem}\label{thm:STinSc} Let $E$ be a finite Galois extension of $K$ with the same field of constants 
$\bF_q$. For any conjugacy class $C$ in $\mathrm{Gal}(E/K)$ and any subinterval $I\subseteq[0, \pi]$, we have  
\begin{align*}
\frac{1}{|\Pi_d(E/K,C)|}&\Big|\{ \fp \in \Pi_d(E/K,C):\, \fp\not\in\mathrm{supp}(a), \,\theta_\fp(a)\in I \} 
\Big|\\
&=\mu_{\mathrm{ST}}(I)+O\Big(q^{-\frac{d}{4}}\sqrt{[E:K]\mathfrak{N}_{a,E/K}}\Big),
\end{align*}
where the $O$-constant is absolute,
and $$\mathfrak{N}_{a,E/K}= N_{a,E/K}\Big(\max(B_a,B_{E/K})+1\Big) +g $$ 
with $N_{a,E/K}:=\sum_{\mathfrak p\in\mathrm{supp}(a)\cup\mathrm{supp}\,\mathfrak D_{E/K}} 
\mathrm{deg}(\mathfrak p)$, $B_a$ being the largest break of the restrictions to the inertia subgroups $I_{\mathfrak p}\; (\mathfrak p\in\mathrm{supp}(a))$
of the Galois representation corresponding to $a^*\mathrm{Kl}(1/2)$, $B_{E/K}$ being the 
largest break of the restrictions to the inertia subgroups $I_{\mathfrak p}\; (\mathfrak p\in\mathrm{supp}\,\mathfrak D_{E/K})$
of all irreducible representations of $\mathrm{Gal}(E/K)$, and $g$ being the genus of $K$. 
\end{theorem}

\begin{remark} Note that $N_{a,E/K},B_a\geqslant 1$ as $a\in K$ is non-constant. 
The result is trivial if $[E:K]$ is not $o(q^{d/2})$. 
\end{remark}

\begin{remark} M. R. Murty and  V. K. Murty \cite[Theorem 1]{MM10} proved a hybrid Chebotarev--Sato--Tate theorem for elliptic curves
over totally real number fields. This was extended recently by Wong \cite{Wo19} to Hilbert modular forms, and an explicit error term was given therein by assuming the Langlands functoriality conjecture and GRH.
\end{remark}

\begin{remark} 
In the context of elliptic curves or holomorphic cusp forms over $\bQ,$ an explicit error term in the Sato--Tate distribution has been obtained by Thorner
\cite{Th21}, building on the recent breakthrough of Newton and Thorne \cite{NT21I, NT21II} on symmetric power functoriality for holomorphic modular forms. Some prior conditional results can be found in \cite{Mu85, Th14,BK16,RT17}. For instance, assuming GRH for all symmetric power $L$-functions for a fixed newform $f$ without CM, Rouse and Thorner \cite{RT17} proved that
\begin{align}\label{eq:Rouse-Thorner}
\frac{1}{x/\log x}|\{p\leqslant x: \lambda_f(p)\in I\}|=\mu_{\mathrm{ST}}(I)+O_f(x^{-\frac{1}{4}}\log x),
\end{align}
where $\lambda_f(p)$ denotes the (normalized) $p$-th Fourier coefficient of $f,$ $I$ is a subinterval of $[-2,2]$ and $\mu_{\mathrm{ST}}$ denotes the Sato--Tate measure on $[-2,2]$. Note that $-\frac{1}{4}$ here can be compared with our exponent $q^{-\frac{d}{4}}$ in Theorem $\ref{thm:STinSc},$ as elaborated below. To conclude the equidistribution, one
amplifies the indicator function of $I$ by linear combinations of symmetric powers, which produces a natural error, say $\delta,$ and Deligne's bound for sums of Frobenius traces would give another term of the shape $\delta^{-1}x^{-\frac{1}{2}}$ up to some harmless factors. Choosing $\delta=x^{-\frac{1}{4}}$  to balance the above two terms, one obtains the exponent $-\frac{1}{4}$ in \eqref{eq:Rouse-Thorner}, and the exponent $q^{-\frac{d}{4}}$ in Theorem $\ref{thm:STinSc}$ appears for the same reason.
\end{remark}

The Chebotarev density theorem is a generalization of Dirichlet's theorem 
on equidistribution of primes in arithmetic progressions. Consequently we shall regard places in $\Pi_d(E/K,C)$ 
as generalization of primes in an arithmetic progression. 

We now specialize to the rational function field $K=\mathbf F_q(t)$. 
Given a polynomial $A$ of degree $d$ and an integer $0\leq h<d$, a \emph{short
interval} around $A$ is defined to be
\begin{equation}
I(A,h):=\{f \in \bF_q[T]: \deg(f-A) \leq h \}.
\end{equation}
In other words, $I(A,h)$ consists of  
elements in $\bF_q[T]$ that are $h$-close to $A$ with respect to the valuation at the place $\infty$. 
Suppose $A$ is monic. Since $h<d$, all polynomials in $I(A,h)$ are monic of degree $d$. Let $\Pi_d(A, h)$ be the 
subset of $I(A, h)$ consisting of prime polynomials with nonzero constant term. Each prime polynomial corresponds to a place of 
$\mathbf F_q(T)$, which we denote by the same notation. Our second result concerns the distribution of Kloosterman sums $\kl(\fp, a)$ as $\fp$ varies in $\Pi_d(A,h)$ 
for small $h$.

\begin{theorem}\label{thm:shortinterval}
Fix $a \in \bF_q(T)-\bF_q$.  
Let $0\leq h<d$ be two nonnegative integers, 
and let $A\in \bF_q[T]$ be a monic polynomial of degree $d$. Assume $$(d,h)\neq (2,1),(3,1),(4,2),\ (5,2).$$ 
For any subinterval $I\subseteq[0, \pi]$, we have  
\begin{eqnarray*}
&&\frac{1}{|\Pi_d(A,h)|} \Big|\{ \fp \in \Pi_d(A,h):\,\fp\not\in\mathrm{supp}(a),\,\theta_\fp(a)\in I\}
\Big|\\
&=&\mu_{\mathrm{ST}}(I)+
O\Big(q^{\frac{d-2h-2}4}N_a^{\frac{1}{2}}(B_a+d-h)^{\frac{1}{2}} \Big),
\end{eqnarray*}
where the $O$-constant is absolute.
\end{theorem}

\begin{remark} The statement is non-trivial only when $h>d/2-1$, in which case one obtains the equidistribution of Kloosterman sums in short intervals. It seems very difficult at present to beat this barrier to obtain similar equidistributions in shorter intervals. Note that Theorem \ref{thm:shortinterval} can be regarded as a high-rank analogue of the Prime Number Theorem in short intervals in the setting of function fields, and the classical (rank-one) case has been successfully treated by Bank, Bary-Soroker and Rosenzweig \cite{BBSR15} in rather short intervals.
\end{remark}

Theorem \ref{thm:STinSc} can be regarded as a joint Chebotarev--Sato--Tate distribution of Kloosterman sums over a function field.  We also have a joint Sato--Tate distribution for a finite family of Kloosterman sums over a function field.

Let $a_i$ $(i=1,\ldots, n)$ be elements in $K$ not lying in $\mathbf F_q$. For any  place $\mathfrak p$ of $K$ outside $\mathrm{supp}(a_1)\cup\cdots\cup \mathrm{supp}(a_n)$, recall that the angles $\theta_{\mathfrak p}(a_i)$ are defined by 
\begin{align*}
\mathrm{Kl}(\mathfrak p,a_i)&= 2 \sqrt{q}\cos\theta_{\mathfrak p}(a_i).
\end{align*}

\begin{theorem}\label{thm:jointdistKl} 
Let $\Pi_d(a_1,\ldots, a_n)$ be the set of places $\mathfrak p$ of $K$ of degree $d$ outside the zeros and poles of $a_i$ $(i=1, \ldots, n)$. Suppose that, for each $i$, there exists a place $\mathfrak p$ which is a zero of $a_i$ but not a zero of $a_j$ for any other $j\not =i$. Then 
for any subintervals $I_1, \ldots, I_n\subset [0,\pi]$, we have
\begin{eqnarray*}
&& \frac{|\{\mathfrak p\in \Pi_d(a_1,\ldots,a_n): 
\theta_{\mathfrak p}(a_i)\in I_i\;(i=1, \ldots, n)\}|}{|\Pi_d(a_1,\ldots, a_n)|} =\mu_{\mathrm{ST}}(I_1)\cdots \mu_{\mathrm{ST}}(I_n)
+O(q^{-\frac{d}{2(n+1)}}).
\end{eqnarray*}
\end{theorem}

Although this paper only deals with the distribution of Kloosterman sums over function fields, our results are valid for local systems with $\mathrm{SL}(2)$ monodromy and suitable ramification restrictions. Moreover, the $n$ families of angles in Theorem \ref{thm:jointdistKl}
don't have to come from the same local system.

\subsection{Organization of this paper}
This paper is organized as follows. In Section 2, we use the works of Deligne and Katz to study the $L$-functions of the sheaves
$\mathrm{Sym}^k\big(a^*\mathrm{Kl}(1/2)\big)$, and other sheaves with similar properties. From these we deduce the estimate
Corollary~\ref{twistedsymk} which is one of the key ingredients for the proofs of our main results.
Theorem \ref{thm:STinSc} is proved in Section 3. 
In Section 4, we prove Theorem \ref{thm:shortinterval} by relating the distribution in a short interval to that in arithmetic progressions. In Section 5, we establish an Erd\H{o}s--Tur\'an inequality for $\mathrm{SU}(2)^{\times n}$ and obtain a criterion of joint Sato--Tate distributions for a finite family of  Galois representations with monodromy groups $\mathrm{SL}(2)$, from which  Theorem \ref{thm:jointdistKl} follows.
In the appendix we prove an effective version of the Chebotarev density theorem that is used in the paper.

\subsection*{List of notation.}
\begin{itemize}[leftmargin=25pt]
\item $\Pi(K)$: the set of places in $K$
\item $\Pi_d(K)$: the subset of places in $\Pi(K)$ with degree $d$
\item $\Pi(E/K)$: the subset of places in $\Pi(K)$ unramified in $E$
\item $\Pi(E/K,C)$: the subset of places $\fp$ in $\Pi(E/K)$ with $\mathrm{Frob}_\fp= C$
\item $\Pi_d(E/K,C):= \Pi_d(K)\cap \Pi(E/K,C)$
\item $\Pi_d$: the set of monic prime polynomials of degree $d$ and nonzero constant term 
\item $I(A,h):=\{f \in \bF_q[T]: \deg(f-A) \leq h \}$
\item $\Pi_d(A, h)$: the subset of $I(A, h)$ consisting of degree $d$ prime polynomials with nonzero constant term
\item $\Pi_d(a_1,..., a_n)$: the subset of places $\mathfrak p$ in $\Pi_d(K)$ outside the zeros and poles of non-constant elements $a_i\in K$ $(i=1,..., n)$
\end{itemize}

\subsection*{Acknowledgements}
The authors would like to thank the  referee for valuable suggestions which improved the presentation of this paper.
LF is supported by 2021YFA1000700, 2023YFA1009703 and NSFC12171261. YKL is supported by GRF (No. 17303619, 17307720) and NSFC (No. 12271458). WCWL was supported by Simons Foundation (No.  355798). 
PX is supported by NSFC (No. 12025106).

\section{Sums of traces of Frobenius} 

\subsection{Basic properties of sums of traces}

Let $K$ be a function field with the constant field $\mathbf F_q$, let $X$ be the geometrically connected 
smooth projective algebraic curve over $\mathbf F_q$
with function field $K$, and let 
$$\rho: \mathrm{Gal}(\overline K/K)\to \mathrm{GL}(V)$$ be a Galois $\overline{\mathbf Q}_\ell$-representation unramified on 
an open subset $U$ of $X$. The Artin conductor of $\rho$ is defined to be the 
divisor 
$$\mathfrak D(\rho)=\sum_{\mathfrak p\in X-U} \Big(\mathrm{dim}\,V-\mathrm{dim}\, V^{I_{\mathfrak p}}
+\mathrm{sw}_{\mathfrak p}(\rho)\Big)\mathfrak p,$$
where $I_{\mathfrak p}$ is the inertia subgroup at $\mathfrak p$, and 
$\mathrm{sw}_{\mathfrak p}(\rho)$ is the Swan conductor of the representation $\rho|_{I_\fp}$. 
By abuse of notation, we also call 
\begin{align}\label{eq:Artinconductor}
a(\rho):=\mathrm{deg}\,\mathfrak D(\rho)=
\sum_{\mathfrak p\in X-U} \Big(\mathrm{dim}\,V-\mathrm{dim}\, V^{I_{\mathfrak p}}+\mathrm{sw}_{\mathfrak p}(\rho)\Big)
\mathrm{deg}(\mathfrak p)
\end{align} 
the Artin conductor of $\rho$. As in the introduction, we fix an isomorphism 
$\iota:\overline{\mathbf Q}_\ell\stackrel\cong\to\mathbf C$ so that we can transform an $\ell$-adic number to a complex number via $\iota$.
We say $\rho$ is \emph{punctually $\iota$-pure of weight $0$} if the sheaf $\mathcal F$ has this property (\cite[1.2.2]{De80}), 
that is, for any place $\mathfrak p$ in $U$, and any eigenvalue $\lambda$ of $\rho(\mathrm{Frob}_{\mathfrak p})$, we have 
$$|\iota(\lambda)|=1.$$ 
We say $\rho$ has \emph{no geometric invariant} (resp. \emph{no geometric coinvariant})
if $V^{\mathrm{Gal}(\overline K/K\overline{\mathbf F}_q)}=0$ (resp. $V_{\mathrm{Gal}(\overline K/K\overline{\mathbf F}_q)}=0$). 

For each positive integer $m$, put
\begin{align}\label{def:S(symk,m)}
S_{\rho}(m):=\sum_{\mathrm{deg}(\mathfrak p)|m }
\mathrm{deg}(\mathfrak p)\,\mathrm{Tr}(\mathrm{Frob}^{\frac{m}{\mathrm{deg}(\mathfrak p)}}_\fp, V^{I_\fp}).
\end{align} 
We have the following bound for $\iota S_{\rho}(m)$.

\begin{lemma}\label{lm:boundforS} 
Suppose $\rho$ is punctually $\iota$-pure of weight $0$, and has neither geometric invariant nor geometric coinvariant. Then we have 
\begin{align*}
|\iota S_{\rho}(m)| &\leqslant q^{m/2} ((2g-2)\dim(\rho)+a(\rho)).
\end{align*}
\end{lemma}

\begin{proof} The representation $\rho$ defines a smooth $\overline{\mathbf Q}_\ell$-sheaf
$\mathcal F$ on $U$ so that the representation $V$ is identified with $\mathcal F_{\bar\eta}$ as $\mathrm{Gal}(\overline K/K)$-modules, 
where $\eta$
is the generic point of $X$. Let $X(\mathbf F_{q^m})$ be the set of $\mathbf F_{q^m}$-points in $X$. Each point
$x\in X(\mathbf F_{q^m})$ lies above a place $\mathfrak p$ with $\mathrm{deg}(\mathfrak p)|m$, and there are exactly 
$\mathrm{deg}(\mathfrak p)$ points in $X(\mathbf F_{q^m})$ lying above $\mathfrak p$. Let $j:U\hookrightarrow X$ 
be the open immersion. We have $(j_*\mathcal F)_{\bar x}\cong V^{I_{\mathfrak p}}$ and 
$$S_{\rho}(m)=\sum_{x\in X(\mathbb F_{q^m})} \mathrm{Tr}(\mathrm{Frob}_x, (j_*\mathcal F)_{\bar x}).$$
 By the Grothendieck trace formula \cite[Rapport 3.2]{De77}, we have
\begin{eqnarray}\label{GPF}
S_{\rho}(m)=\sum_{i=0}^2(-1)^i \mathrm{Tr}(F^m, H^i(X\otimes_{\mathbf F_q}\overline{\mathbf F}_q, j_*\mathcal F)),
\end{eqnarray} 
where $F$ is the Frobenius correspondence, and
$H^i(\hbox{-,-})$ denote the $\ell$-adic cohomology groups. 
By \cite[1.4.1]{De80}, we have
\begin{eqnarray*}
H^0(X\otimes_{\mathbf F_q}\overline{\mathbf F}_q, j_*\mathcal F)&\cong& V^{\mathrm{Gal}(\overline K/K\overline{\mathbf F}_q)},\\
H^2(X\otimes_{\mathbf F_q}\overline{\mathbf F}_q, j_*\mathcal F)&\cong& V_{\mathrm{Gal}(\overline K/K\overline{\mathbf F}_q)}(-1).
\end{eqnarray*}
They vanish by our assumption. By (\ref{GPF}), we have
$$S_{\rho}(m)=-\mathrm{Tr}(F^m, H^1(X\otimes_{\mathbf F_q}\overline{\mathbf F}_q, j_*\mathcal F)).$$
By the Grothendieck--Ogg--Shafarevich formula (\cite[X 7.12]{Gr77}) for the Euler characteristic,
we have $$\mathrm{dim}\, H^1(X\otimes_{\mathbf F_q}\overline{\mathbf F}_q, j_*\mathcal F))=(2g-2)\dim(\rho)+a(\rho).$$
By Deligne's theorem \cite[3.2.3]{De80} (the Weil conjecture), for any eigenvalue $\lambda_i$ of $F$ on 
$H^1(X\otimes_{\mathbf F_q}\overline{\mathbf F}_q, j_*\mathcal F)$, we have
$$|\iota(\lambda_i)|=q^{1/2}.$$
Our estimate for $S_{\rho}(m)$ follows. 
\end{proof}

The above lemma yields the following key estimate for sums of Frobenius traces.

\begin{proposition}\label{prop:Frobeinus-keyestimate} 
Suppose $\rho:\mathrm{Gal}(\overline K/K)\to \mathrm{GL}(V)$ is unramified on $U$, punctually $\iota$-pure of weight $0$, 
and has neither geometric invariant nor geometric coinvariant.

(i) For each positive integer $m,$ we have
\begin{align*}
\Bigg|\sum_{\fp\in U,\,\mathrm{deg}(\mathfrak p) = m} \iota \mathrm{Tr}(\mathrm{Frob}_{\mathfrak p}, V)\Bigg|
\leq
\frac{q^{m/2}}m \Big((6g+1)\dim(\rho)+a(\rho)\Big)+\omega \dim(\rho),
\end{align*}
where
$\omega$ is the number of distinct places in $X-U$.

(ii) Let $B$ be the largest break
of the representations $\rho|_{I_\fp}$ for all $\mathfrak p\in X-U$, and let
$$N=\sum_{\mathfrak p\in X-U} \mathrm{deg}(\mathfrak p)$$ be the number of points in
$(X-U)\otimes_{\mathbf F_q} \overline{\mathbf F}_q$. We have 
\begin{align*}
\Bigg|\sum_{\fp\in U,\,\mathrm{deg}(\mathfrak p) = m} \iota \mathrm{Tr}(\mathrm{Frob}_{\mathfrak p}, V)\Bigg|
\leq
\frac{q^{m/2}}m (6g+1+N(B+3))\mathrm{dim}(\rho),
\end{align*}

\end{proposition}

\begin{proof} 
(i) Recalling the definition (\ref{def:S(symk,m)}) of $S_{\rho}(m)$, we have
\begin{align*}
\sum_{\fp\in U,\,\mathrm{deg}(\mathfrak p)=m}\mathrm{Tr}(\mathrm{Frob}_{\mathfrak p}, V)
&=\frac{S_{\rho}(m)}{m}-
\sum_{\fp\in X-U,\,\mathrm{deg}(\mathfrak p)=m}
\mathrm{Tr}(\mathrm{Frob}_\fp, V^{I_\fp})\\
&\ \ \ \ -\frac{1}{m}\sum_{\mathrm{deg}(\mathfrak p)|m, \, \mathrm{deg}(\mathfrak p)\not=m}
\mathrm{deg}(\mathfrak p)\mathrm{Tr}(\mathrm{Frob}^{\frac{m}{\mathrm{deg}(\mathfrak p)}}_\fp,V^{I_\fp}).
\end{align*}
By \cite[1.8.1]{De80}, we have
$$|\iota \mathrm{Tr}(\mathrm{Frob}^r_\fp, V^{I_\fp})|\leq \mathrm{dim}(\rho)$$ for all $r$. 
Hence
\begin{align*}
\Big|\sum_{\fp\in X-U,\,\mathrm{deg}(\mathfrak p)=m}
\iota \mathrm{Tr}(\mathrm{Frob}_\fp,V^{I_\fp})\Big|\leq  \omega\,\mathrm{dim}(\rho).
\end{align*}
and
\begin{align*}
\Big|\sum_{\mathrm{deg}(\mathfrak p)|m, \, \mathrm{deg}(\mathfrak p)\not=m}
\mathrm{deg}(\mathfrak p)\iota \mathrm{Tr}\Big(\mathrm{Frob}^{\frac{m}{\mathrm{deg}(\mathfrak p)}}_\fp, V^{I_\fp}\Big)\Big|
&\le\mathrm{dim}(\rho)\sum_{\mathrm{deg}(\mathfrak p)|m,  \mathrm{deg}(\mathfrak p)\not=m}\mathrm{deg}(\mathfrak p).
\end{align*}
In view of Lemma \ref{lm:boundforS}, it suffices to prove
\begin{align}\label{estimate1}
\sum_{\mathrm{deg}(\mathfrak p)|m,  \mathrm{deg}(\mathfrak p)\not=m}\mathrm{deg}(\mathfrak p)
&\leqslant (4g+3)q^{m/2}.
\end{align}

Note that $\sum_{\mathrm{deg}(\mathfrak p)|r} \mathrm{deg}(\mathfrak p)$
is the number of $\mathbf F_{q^r}$-points of $X$. By the Riemann hypothesis for $X$ proved by Weil, we have
\begin{eqnarray}\label{estimate2}
\sum_{\mathrm{deg}(\mathfrak p)|r} \mathrm{deg}(\mathfrak p)=1+q^r -\sum_{i=1}^{2g} \lambda_i^r,
\end{eqnarray}
where $\lambda_i$ are complex numbers with $|\lambda_i|=q^{1/2}$. 
In particular, we have 
\begin{align*}
\sum_{\deg \mathfrak{p} = r}\mathrm{deg}(\mathfrak p)
\leq q^r + 2g q^{r/2} +1.
\end{align*}
It can be easily seen that
\begin{align*}
\sum_{r|m,\, r\not=m} q^r
\leqslant \sum_{1\leqslant j\leqslant \frac{m}{2}} q^j
\leqslant \frac{q^{\frac{m}{2}+1}-q}{q-1}
\leqslant 2q^{m/2},
\end{align*}
from which we derive
\begin{align*}
\sum_{\mathrm{deg}(\mathfrak p)|m,  \mathrm{deg}(\mathfrak p)\not=m}\mathrm{deg}(\mathfrak p)
&\leq\sum_{r|m,\, r\not=m} (q^r + 2g q^{r/2} +1)\\
&\leq\sum_{r|m,\, r\not=m} \Big(q^r + 2g q^r +\frac{1}{2}q^r\Big)\\
&\leq \Big(2g+\frac{3}{2}\Big)\sum_{r|m,\, r\not=m} q^r \leq (4g+3)q^{m/2}.
\end{align*}

(ii) Note that 
\begin{align*}
a(\rho)
&=\sum_{\mathfrak p\in X-U} (\mathrm{dim}\,V-\mathrm{dim}\, V^{I_{\mathfrak p}}+\mathrm{sw}_{\mathfrak p}(\rho))
\mathrm{deg}(\mathfrak p)\\
&\leq\sum_{\fp\in X-U} (\mathrm{dim}(\rho)+ B\, \mathrm{dim}(\rho) )\mathrm{deg}(\mathfrak p)\\
&=N(B +1)\mathrm{dim}(\rho).
\end{align*} 
On the other hand, we have
$$\blue{\omega}\leq N\leq \frac{q^{m/2}}{m} \, 2N.$$
The assertion then follows from (i). 
\end{proof}

Let $\Pi_m(K)$ be the set of places $\mathfrak p$ of $K$ such that $\mathrm{deg}(\mathfrak p)=m$. By (\ref{estimate2}), we have
 \begin{eqnarray*}
 m|\Pi_m(K)|+\sum_{\mathrm{deg}(\mathfrak p)|m,\, \mathrm{deg}(\mathfrak p)\not=m} 
 \mathrm{deg}(\mathfrak p)=1+q^m -\sum_{i=1}^{2g} \lambda_i^m,
\end{eqnarray*}
 Combined with the estimate (\ref{estimate1}) and $|\lambda_i|=q^{1/2}$, we get 
 the following analogue of the prime number theorem for function fields that will be used later.

\begin{lemma}\label{lm:PNT}
Let $\Pi_m(K)$ be the set of places $\mathfrak p$ of $K$ such that $\mathrm{deg}(\mathfrak p)=m$. Then
we have $$\Big| |\Pi_m(K)|-\frac{q^m}{m}\Big|\leq  (6g+4)\frac{q^{m/2}}{m}.$$
\end{lemma}

\subsection{Symmetric powers of the Kloosterman sums}
Let $a\in K-\mathbf F_q$.
As explained in the introduction, for any $\fp\not\in\mathrm{supp}(a)$, we have 
\begin{align}\label{eq:klsumangle}
\kl(\fp,a)=2 \sqrt{N\fp}\cos \theta_\fp(a) 
\end{align} for some $\theta_\fp(a)\in [0, \pi]$. 
Let $V=(a^*\mathrm{Kl}(1/2))_{\bar \eta}$, and let 
$$\rho_a:\mathrm{Gal}(\overline K/K)\to \mathrm{GL}(V)$$ be the Galois $\overline{\mathbf Q}_\ell$-representation of dimension 2 corresponding to
$a^*\mathrm{Kl}(1/2)$ which is unramified on $U:=X-\mathrm{supp}(a)$.
The smooth 
$\overline{\mathbf Q}_\ell$-sheaf $a^*\mathrm{Kl}(1/2)$ on $U$ is punctually $\iota$-pure of weight $0$ by the construction of the Kloosterman sheaf. 
So is the $k$-th symmetric power $\Sym^k\rho_a$ of $\rho_a$ for all $k$. By \cite[11.1]{Ka88}, the geometric monodromy group of 
$\mathrm{Kl}(1/2)$ is $\SL_2$. Since $\SL_2$ has no proper subgroup of finite index, the 
geometric monodromy group of $a^*\mathrm{Kl}(1/2)$ is also $\SL_2$.

For any place $\mathfrak p\not\in\mathrm{supp}(a)$ of $K$, we have
$$\iota\mathrm{Tr}(\rho_a(\mathrm{Frob}_\fp), V)=-\frac{\kl(\fp,a)}{\sqrt{N\fp}}.$$ 
Moreover, we have 
\begin{align*}
\iota \mathrm{det}(1-\mathrm{Frob}_\fp t^{\mathrm{deg}(\mathfrak p)}, V)
&=1 + \frac{\kl(\fp,a)}{\sqrt{N\fp}} t^{\mathrm{deg}(\mathfrak p)} + t^{2\mathrm{deg}(\mathfrak p)}\\
&=(1 + e^{i\theta_\fp(a)}t^{\mathrm{deg}(\mathfrak p)})(1+ e^{-i\theta_\fp(a)}t^{\mathrm{deg}(\mathfrak p)}).
\end{align*}  
For the representation 
$\Sym^k\rho_a$ and $\mathfrak p\not\in\mathrm{supp}(a)$, we have 
\begin{align*}
\iota \mathrm{Tr}(\mathrm{Frob}_\fp, \mathrm{Sym}^k V)=(-1)^k\sum_{0 \le j \le k} (e^{i\theta_\fp(a)})^{k-j}(e^{-i\theta_\fp(a)})^j.
\end{align*}
Define
\begin{align*}
\Sym^k (\theta):= \sum_{0 \le j \le k} (e^{i\theta})^{k-j}(e^{-i\theta})^j,
\end{align*}
so that 
\begin{eqnarray*}
\iota \mathrm{Tr}(\mathrm{Frob}_\fp, \mathrm{Sym}^k V) = (-1)^k \Sym^k(\theta_\fp(a)).
\end{eqnarray*}

Let $E/K$ be a finite Galois extension and let $\mathfrak D_{E/K}$ be its discriminant. 
Then $\mathrm{supp}\,\mathfrak D_{E/K}$ consists of places of $K$ ramified in $E$. Let 
$\sigma: \mathrm{Gal}(E/K)\to\mathrm{GL}(W)$ be a Galois $\overline{\mathbf Q}_\ell$-representation. 
The Galois representation $(\Sym^k\rho_a)\otimes \sigma$ is unramified on $U:=X-\mathrm{supp}(a)\cup\mathrm{supp}\,\mathfrak D_{E/K}$.
Instead of the character $\mathrm{Tr}(\sigma)$ of $\sigma$, we work with $\chi_\sigma=\iota\mathrm{Tr}(\sigma)$. We will estimate  
the twisted moment
\begin{align*}
\sum_{\substack{\mathrm{deg}(\mathfrak p) = m,\\
\fp\not\in\mathrm{supp}(a)\cup\mathrm{supp}\,\mathfrak D_{E/K}}} \Sym^k(\theta_\fp(a))\chi_\sigma(\mathrm{Frob}_{\fp}).
\end{align*} 

\begin{lemma}\label{lm:symtensorsigma}
The representation $(\Sym^k\rho_a)\otimes \sigma$ is punctually $\iota$-pure of weight $0$,
and has neither geometric invariant nor geometric coinvariant.
\end{lemma}

\begin{proof} The Galois representations $\sigma$ is punctually $\iota$-pure of weight $0$
since it factors through the finite group $\mathrm{Gal}(E/K)$. Since $\Sym^k\rho_a$ is also 
punctually $\iota$-pure of weight $0$, the representation
$(\Sym^k\rho_a)\otimes \sigma$ has the same property. Since the geometric monodromy group 
of $\rho_a$ is $\SL_2$ by \cite[11.1]{Ka88} and $\SL_2$ has no proper subgroup of finite index, 
the Zariski closure of the image of $\rho_a|_{\mathrm{Gal}(\overline K/E\overline{\mathbf F}_q)}$ is also $\SL_2$.
In particular, $\Sym^k\rho_a|_{\mathrm{Gal}(\overline K/E\overline{\mathbf F}_q)}$ is a nontrivial irreducible representation
and it has neither invariant nor coinvariant. Since $\mathrm{Gal}(\overline K/E\overline{\mathbf F}_q)$ acts trivially on $W$, we have 
$$((\mathrm{Sym}^kV)\otimes W)^{\mathrm{Gal}(\overline K/E\overline{\mathbf F}_q)}\cong
(\mathrm{Sym}^k V)^{\mathrm{Gal}(\overline K/E\overline{\mathbf F}_q)}\otimes W=0.$$ 
Since $$((\mathrm{Sym}^kV)\otimes W)^{\mathrm{Gal}(\overline K/K\overline{\mathbf F}_q)}
\subset ((\mathrm{Sym}^kV)\otimes W)^{\mathrm{Gal}(\overline K/E\overline{\mathbf F}_q)},$$
we must have $((\mathrm{Sym}^kV)\otimes W)^{\mathrm{Gal}(\overline K/K\overline{\mathbf F}_q)}=0$ and hence 
$(\Sym^k\rho_a)\otimes \sigma$ has no geometric invariant. Similarly, it has no geometric coinvariant. 
\end{proof}

Thanks to Lemma \ref{lm:symtensorsigma}, we are in a good position to apply Proposition \ref{prop:Frobeinus-keyestimate} (ii)
with $\rho=(\Sym^k\rho_a)\otimes \sigma$. This allows us to derive the following estimate.

\begin{corollary}\label{twistedsymk}
Keep the notation and conventions as above. For each positive integer $k,$ we have
\begin{align*}
 &\Bigg|\sum_{\substack{\mathrm{deg}(\mathfrak p) = m,\\
\fp\not\in\mathrm{supp}(a)\cup\mathrm{supp}\,\mathfrak D_{E/K}}} \Sym^k(\theta_\fp(a))\chi_\sigma(\mathrm{Frob}_{\fp})\Bigg|\\
&\leq \Big(6g+1+N_{a,E/K}\big(\max(B_a, B_{E/K})+3\big)\Big)(k+1)\mathrm{dim}(\sigma)\frac{q^{m/2}}m\\
&\leq 6(k+1)\fN_{a,E/K}\dim(\sigma)\cdot \frac{q^{m/2}}m,
\end{align*} 
where $N_{a, E/K}$ is  the number of points in $(\mathrm{supp}(a)\cup\mathrm{supp}\,\mathfrak D_{E/K})\otimes_{\mathbf F_q}\overline{\mathbf F}_q$, and $B_a$ (resp. $B_{E/K}$) is the largest break of the representations $\rho_a|_{I^\fp}$
(resp. $\rho|_{I_{\mathfrak p}}$) for all $\mathfrak p\in 
\mathrm{supp}(a)\cup\mathrm{supp}\,\mathfrak D_{E/K},$ and
\begin{align*}
\fN_{a,E/K}=g+N_{a,E/K}\Big(\max(B_a,B_{E/K})+1\Big).
\end{align*}
\end{corollary}

\smallskip

\section{Equidistribution of Kloosterman sums in arithmetic progressions}\label{sect3}

Given a sequence of angles $\{\theta_i\}\subseteq[0,\pi],$ we would like to examine the proportion of these elements lying in a prescribed interval $I\subseteq[0,\pi]$. In the spirit of Weyl, it is necessary to develop a Fourier expansion of the characteristic function  $\chi_I$ of $I$. In the situation of Sato--Tate distributions, we recall the following truncated approximation which is due to Rouse and Thorner \cite[Lemma 3.1]{RT17}.

\begin{lemma}\label{lm:NRT}
Let $I\subseteq[0,\pi]$ be an interval and $\theta\in[0,\pi].$
For each positive integer $L$, there exist real numbers $a_k^\pm$ $(1\leqslant k\leqslant L)$ and $\delta^\pm(\theta)$ such that $$a_k^\pm\ll k^{-1},\quad \delta^\pm(\theta) \ll L^{-1},$$ and
\begin{align*}
\chi^-(\theta)\leqslant \chi_I(\theta)\leqslant \chi^+(\theta)
\end{align*}
with
\begin{align*}
\chi^\pm(\theta)=\mu_{\mathrm{ST}}(I)+\sum_{1\leqslant k\leqslant L ~}a_k^\pm\Sym^k(\theta)+\delta^\pm(\theta).
\end{align*}
Moreover, we have the bound $\|\chi^\pm\|_\infty\ll1.$
All the constants implied by $\ll$ are absolute.
\end{lemma} 

\begin{remark}
The bound $\|\chi^\pm\|_\infty\ll1$ was implicitly contained in \cite[Lemma 3.1]{RT17}. One may prove an explicit (but crude) bound 
$|\chi^\pm(\theta)|\leqslant 2\pi$ by appealing to \cite[Chapter 1, Section 2]{Mo94} and \cite{BMV01}.
\end{remark}

We use the following quantitative version of Weyl's criterion for the Sato--Tate distribution due to Niederreiter \cite[Lemma 3]{Ni91}. For completeness, we deduce it from the above Lemma \ref{lm:NRT}. 

\begin{lemma}[Niederreiter]\label{lm:Niederreiter} Let $\{\theta_\lambda\}_{\lambda\in \Lambda}$ be a finite family of angles in 
$[0, \pi]$. For any integer $L\geq 1$ and  any subinterval $I\subseteq [0,\pi]$, we have 
\begin{eqnarray*}
\frac{|\{\lambda\in \Lambda: \theta_\lambda\in I\}|}{|\Lambda|}-\mu_{\mathrm{ST}}(I)\ll \frac{1}{|\Lambda|} \sum_{1\leqslant k\leqslant L}\frac{1}{k}\Big|\sum_{\lambda\in \Lambda}\mathrm{Sym}^k(\theta_\lambda)\Big|+\frac{1}{L}.
\end{eqnarray*}
\end{lemma}

\begin{proof} Keep the notation in Lemma \ref{lm:NRT}. Firstly,
\begin{align*}
\frac{|\{\lambda\in \Lambda: \theta_\lambda\in \Lambda\}|}{|\Lambda|}=
\frac{1}{|\Lambda|}\sum_{\lambda\in\Lambda} \chi_I(\theta_\lambda)
\leq \frac{1}{|\Lambda|}\sum_{\lambda\in\Lambda} \chi_I^+(\theta_\lambda).
\end{align*}
By Lemma \ref{lm:NRT}, we may write
\begin{align*}
\frac{1}{|\Lambda|}\sum_{\lambda\in\Lambda} \chi_I^+(\theta_\lambda)
&=\frac{1}{|\Lambda|}\sum_{\lambda\in\Lambda} \Big(\mu_{\mathrm{ST}}(I)+\sum_{1\leqslant k\leqslant L}
a_k^+\mathrm{Sym}^k(\theta_\lambda)+\delta^+\Big)\\
&=\mu_{\mathrm{ST}}(I)+\frac{1}{|\Lambda|}\sum_{1\leqslant k\leqslant L}
a_k^+\sum_{\lambda\in\Lambda}\mathrm{Sym}^k(\theta_\lambda)+\delta^+\end{align*}
with $a_k^+\ll k^{-1}$ and $\delta^+\ll L^{-1}$.
This gives the desired upper bound. We get the lower bound by using $\chi^-$.
\end{proof}

\subsection{Proof of Theorem \ref{thm:STinSc}}
By orthogonality, for every positive integer $k\geq 1$, we have
\begin{align*}
\sum_{\substack{\fp \in \Pi_d(E/K,C)\\ \fp\not\in\mathrm{supp}(a)}}\Sym^k(\theta_\fp(a))
=\frac{|C|}{[E:K]}\sum_{\sigma} \overline{\chi_\sigma}(C)\sum_{\substack{\mathrm{deg}(\mathfrak p) = m,\\
\fp\not\in\mathrm{supp}(a)\cup\mathrm{supp}\,\mathfrak D_{E/K}}} \Sym^k(\theta_\fp(a))\chi_\sigma(\mathrm{Frob}_{\fp}),
\end{align*}
where $\sigma$ runs over all complex irreducible representations of $\mathrm{Gal}(E/K)$.
From Corollary \ref{twistedsymk}, the fact that $|\chi_\sigma(C)|\leq \dim(\sigma)$, and the first equality in Lemma \ref{Hurwitz}, it follows that
\begin{align*}
\Big|\sum_{\substack{\fp \in \Pi_d(E/K,C)\\\fp\not\in\mathrm{supp}(a)}}\Sym^k(\theta_\fp(a))
\Big|
&\leq 6(k+1)\mathfrak N_{a,E/K}\frac{q^{d/2}}{d} \frac{|C|}{[E:K]}\sum_{\sigma} \mathrm{dim}(\sigma)^2\\
&=6(k+1)|C|\mathfrak N_{a,E/K} \frac{q^{d/2}}{d}.
\end{align*}
Applying Lemma \ref{lm:Niederreiter} to the case  
\begin{align*}
\Lambda = \{\fp \in \Pi_d(E/K,C):\fp\not\in\mathrm{supp}(a)\},\ \ \{\theta_\lambda\}_{\lambda\in \Lambda}=\{\theta_\fp(a)\}_{\fp \in \Pi_d(E/K,C),\,\fp\not\in\mathrm{supp}(a)},
\end{align*}
we obtain 
\begin{align*}
&\ \ \ \ \frac{1}{|\Lambda|}\Big|\{ \fp \in \Pi_d(E/K,C):\, \fp\not\in\mathrm{supp}(a), \,\theta_\fp(a)\in I\}\Big|
-\mu_{\mathrm{ST}}(I)\\
&\ll \frac{1}{|\Lambda|}\sum_{1\leqslant k\leqslant L}\frac{1}{k}\Big|\sum_{\substack{\fp \in \Pi_d(E/K,C)\\ \fp\not\in\mathrm{supp}(a)}}\Sym^k(\theta_\fp(a))\Big|+\frac{1}{L}\\
&\ll \frac{|C|\mathfrak N_{a,E/K}}{|\Lambda|} \frac{q^{d/2}}{d}L+\frac{1}{L}.
\end{align*}
By Theorem \ref{thm:effChebo} (iii),  we have 
\begin{align*}
|\Pi_d(E/K, C)|=\frac{|C|}{[E:K]}\frac{q^d}{d}\Big(1+ O\big([E:K]\mathfrak{N}_{a,E/K} q^{-\frac{d}{2}}\big)\Big).
\end{align*}
Since $|\Pi_d(E/K,C)|-|\Lambda|\leqslant \fN_{a,E/K},$
we have $$|\Lambda|\geq\frac{|C|}{2[E:K]}\frac{q^d}{d}$$ for large $d$. We thus have 
\begin{align*}
\frac{|C|\mathfrak N_{a,E/K}}{|\Lambda|} \frac{q^{d/2}}{d}\ll [E:K]\mathfrak N_{a,E/K}q^{-d/2}
\end{align*}
and
\begin{align*}
\frac{1}{|\Lambda|}\Big|\{ \fp \in \Pi_d(E/K,C):\, \fp\not\in\mathrm{supp}(a), \,\theta_\fp(a)\in I\}\Big|
-\mu_{\mathrm{ST}}(I)
\ll [E:K]\mathfrak N_{a,E/K}q^{-d/2}L+\frac{1}{L}.
\end{align*}
Taking $L$ to be the integer closest to $([E:K]\mathfrak N_{a,E/K}q^{-\frac{d}{2}})^{-\frac{1}{2}}$, we complete the proof.

\subsection{Rational function field}

We take a closer look at the sets $\Pi_d(E/K,C)$ for the case that $K= \bF_q(T)$ is a rational function field. 
The places of $K$ are in one-to-one correspondence with closed points on the projective line $\mathbb P^1$ over $\bF_q$. 
The place $\infty$ corresponds to the infinity on $\mathbb P^1$ arising from the zero of $1/T$. 
The other places correspond to closed points on the affine line $\mathbb A^1=\mathrm{Spec}\,\mathbf F_q[T]$, 
each of which is an ideal of $R:=\mathbf F_q[t]$ generated by a monic irreducible polynomial.
These polynomials and $1/T$ are uniformizers of the corresponding places. We shall denote a place by the chosen uniformizer.

Carlitz \cite{Ca35, Ca38} introduced and studied an action of $R$ on the
algebraic closure $\overline K$ of $K$, known as the Carlitz module. Elements in $\bF_q$ act as scalar multiplication and 
$T$ sends $u \in \overline K$ to $Tu + u^q$. Together they generate the actions of $R$ by composition and linearity. Consequently, 
a monic polynomial $g(T) \in R$ of degree $d$ sends $u \in \overline K$ to $g(T)(u)$ which is 
a polynomial in $R[u]$, with the lowest degree term $g(T)u$ and the highest degree term $u^{q^d}$. The splitting field of $g(T)(u)$, that is, the extension of $K$ by joining the $q^d$ roots of $g(T)(u)$ in $\overline K$, the so-called $g$-torsion points, is a finite Abelian extension $K_g$ of $K$ with $\bF_q$ as its field of constants  and we have an isomorphism 
$$(R/(g))^\times\stackrel\cong\to \mathrm{Gal}(K_g/K)$$ 
induced from the action of $R$ on $K_g$, where $(R/(g))^\times$ is the group of units of the quotient ring $R/(g)$. 
The extension $K_g / K$ is unramified outside $\infty$ and the closed points where $g(T)$ vanishes.  Thus a place $\fp$ of $K$ unramified in $K_g$ can be identified with a monic  irreducible polynomial $\fp$ in $R$ coprime to $g$. The Frobenius $\mathrm{Frob}_{\fp}$ in $\mathrm{Gal}(K_g/K)$ is the image of 
the polynomial $\fp$ in $(R/(g))^\times$. Given $c \in (R/(g))^\times$, the set $\Pi_d(K_g/K,c)$ consists of monic irreducible polynomials in $R$
of degree $d$ congruent to $c$ modulo $g$. We conclude from \cite[(3)]{BBSR15} and Lemma \ref{lm:PNT} that
$$|\Pi_d(K_g/K,c)|= \frac{q^d}{d[K_g:K]}  +O\Big(\frac{q^{d/2}}{d}\Big( \deg g +\frac{1}{[K_g:K]}\Big)\Big).$$

We apply Theorem~\ref{thm:STinSc} to the case where $K = \bF_q(T)$ and $E = K_{T^r}$ for an integer $r \ge 1$. The Galois group 
$\mathrm{Gal}(E/K)$ is isomorphic to $(\bF_q[T]/(T^r))^\times$. 
The field $E$ is totally ramified at the place $0$ with ramification index $q^r - q^{r-1} = [E : K]$. Hayes \cite{Ha74} proved that there are $q^{r-1}$ places in $E$ above $\infty$, each being tamely ramified over $\infty$ with ramification index $q-1$. 
Any character of $\mathrm{Gal}(E/K)$ is unramified outside $0$ and $\infty$, and its break is at most $r$ at $0$ by class field theory, 
and $0$ at $\infty$ since the ramification is tame. Let $N_a=\sum_{\mathfrak p\in\mathrm{supp}(a)} \mathrm{deg}(\mathfrak p)$, which is $\ge 1$ for non-constant $a$. We then have 
$$B_{E/K}=r, \quad N_{a, E/K}\leq N_a+2.$$ and 
the factor $\mathfrak{N}_{a,E/K}$ in Theorem~\ref{thm:STinSc} satisfies
\begin{align*}
\mathfrak{N}_{a,E/K}\ll N_a\max(B_a,r).
\end{align*}

Given any positive integer $d$ and any polynomial $c \in \bF_q[T]$ with nonzero constant term, the set $\Pi_d(K_{T^r}/K,c)$ consists of monic irreducible polynomials $\fp$ of degree $d$ in $\bF_q[T]$ which are congruent to $c$ modulo $T^r$. Theorem \ref{thm:STinSc} for this case reads as follows.

\begin{theorem}\label{thm:STinScKrational}
Suppose $K = \bF_q(T),$ $E = K_{T^r}$ and $a \in K-\bF_q$. 
Let $c \in \bF_q[T]$ be a polynomial coprime to $T$ and of degree $<r.$ 
Given any subinterval $I\subseteq[0, \pi]$ and any positive integer $d\geq 1,$ we have
\begin{align*}
& \frac{1}{|\Pi_d(E/K,c)|}|\{ \fp \in \Pi_d(E/K,c):\, \fp\not\in\mathrm{supp}(a),\, \theta(\fp,a)\in I\}|\\
=&\mu_{\mathrm{ST}}(I)+O\Big(q^{\frac{r}{2}-\frac{d}{4}}N_a^{\frac{1}{2}}(B_a+r)^{\frac{1}{2}}\Big),
\end{align*}
where  the $O$-constant is absolute.
\end{theorem}

\section{Sato--Tate distribution in short intervals}

In this section we take $K=\bF_q(T)$. Denote by $\cP_d$ the set of polynomials of degree $d$
in $\mathbf F_q[t]$, $\cM_d$ the subset of monic polynomials in $\cP_d$, and $\Pi_d$ the subset of prime 
polynomials with nonzero constant term in 
$\cM_d$. For $d\ge 2$, $\Pi_d$ is the set of  prime polynomials in $\cM_d$, but  $\Pi_1=\{ T+\mu: \mu\in \bF_q^\times\}$ 
misses the prime polynomial $T$. Fix a monic polynomial $A$ of degree $d$ and an integer $0\leq h<d$. Recall that 
$I(A, h)$ consists of those polynomials $f$ such that $\mathrm{deg}(A-f)\leq h.$ Let $\Pi_d(A,h)=I(A, h)\cap \Pi_d$. 
We begin by seeking a more amenable description of the set $\Pi_d(A,h)$.

\subsection{An involution} For $0\neq f\in \bF_q[T]$ we define
\begin{equation*}
  f^*(T):=T^{\deg f} f\Big(\frac 1T\Big).
\end{equation*}
In other words, if  $f(T) =a_0+a_1T+\dots +a_dT^d$ with $\deg f=d$ (so that
$a_d\neq 0$), then $f^*$ is the ``reversed'' polynomial
\begin{equation*}
  f^*(T) = a_0T^d+a_1T^{d-1}+\dots +a_d.
\end{equation*}
We always have $f^*(0)\not=0$. 
Note that  $f(0)\neq 0$ if and only if $\deg f^*
= \deg f$. Moreover $^*$ is an involution on polynomials which do not vanish at
$0$, equivalently, coprime to $T$:
\begin{equation*}
f^{**} = f ~ {\rm if} ~  f(0)\neq 0.
\end{equation*}
We have multiplicativity
\begin{equation*}
  (fg)^* = f^* g^*.
\end{equation*}
Moreover, if $f(0)\neq 0$, we have
\begin{equation*}\label{eq3.9}
f^*(0)^{-1}f\in\Pi_d\Longleftrightarrow f(0)^{-1}f^*\in\Pi_d.
\end{equation*}
Using the involution $^*$, one may convert short intervals to arithmetic progressions.
This key observation is contained in \cite{KR14}. Here we present a proof for the sake of completeness.

\begin{lemma}\label{congruence}
Let $0\leq h<d$ be two integers, let $f\in\cP_d$ with $f(0)\neq0$, and let $B\in\cP_{d-h-1}.$ Then we have $\deg f^*=d,$ and 
\begin{equation*}
f\in I(T^{h+1}B,h) \Longleftrightarrow f^* \equiv B^*  \bmod {T^{d-h}} .
\end{equation*}

\end{lemma}

\proof We have $f= T^{h+1}B+g$ with $\mathrm{deg}\,g \leq h$ if and only if $f^* = B^*
+T^{d-h}g^* $. The lemma follows immediately.
\endproof

Next we estimate the cardinality of $\Pi_d(A,h)$.

\begin{lemma}\label{lm:PNT-short} 
We have $|\Pi_d(A,0)| = q-1$ if $d=1$, and $$
| \Pi_d(A,h)| =\frac{q^{h+1}}d +O(q^{d/2+1})$$ if $d\ge 2.$
\end{lemma}

\proof  
The case of $d=1$ and $h=0$ is clear since $$\Pi_d(A,0)=\{T+\mu: \mu\in \bF_q^\times\}.$$  
For $d\ge 2$, let $[\Pi_d]$ be the set of irreducible (but not necessarily monic) polynomials in $\cP_d$ and 
for $\lambda,\mu\in \bF_q^\times$, let
$$
[\Pi_d]_{\lambda,\mu}:=\{f\in [\Pi_d]: f^*(0)=\lambda,f(0)=\mu\}.
$$ 
Note that $\lambda$ and $\mu$ correspond to the leading coefficient and the constant term of $f$ respectively. We have
$\Pi_d=\bigsqcup_{\mu\in \bF_q^\times} [\Pi_d]_{1,\mu}$. 

Let $B\in \cM_{d-h-1}$ be the unique monic polynomial determined by the congruence 
$$A^* \equiv B^* \bmod {T^{d-h}}.$$
\noindent Then 
$$I(A,h)=I(T^{h+1}B,h).$$
For any $\mu\in \bF_q^\times$, 
by Lemma~\ref{congruence}, 
we have a bijection 
$$f\mapsto f(0)^{-1}f^*$$ from the set $I(A,h) \cap [\Pi_d]_{1,\mu}$ onto the set
$$
\{g\in [\Pi_d]_{1,\mu^{-1}}: g \equiv \mu^{-1} B^* \bmod {T^{d-h}}\}.
$$
Since $B^*(0)=1$, the last set is the same as  
$$
\{g\in \Pi_d: g \equiv \mu^{-1} B^* \bmod {T^{d-h}}\},
$$ 
whose cardinality, denoted by $\pi_q(d,T^{d-h}, \mu^{-1}B^*)$, satisfies 
$$
\pi_q(d; T^{d-h}, c^{-1}B^*) = \frac{|\Pi_d|}{\Phi(T^{d-h})} +O\Big(\frac{q^{d/2}}d(d-h)\Big) 
$$ by \cite[(3)]{BBSR15},
where $$\Phi(T^{d-h})=| (\bF_q[T]/(T^{d-h}))^\times|= q^{d-h}(1-q^{-1}),$$
from which and Lemma \ref{lm:PNT}, we infer  
$$
|I(A,h) \cap [\Pi_d]_{1,\mu}| =  \frac{q^d}{d \Phi(T^{d-h})} +O(q^{d/2})=\frac{q^{h+1}}{d (q-1)} +O(q^{d/2}).
$$
It follows that
$$
|\Pi_d(A,h)|=\sum_{\mu\in \bF_q^\times} | I(A,h)\cap [\Pi_d]_{1,\mu}|=\frac{q^{h+1}}d +O(q^{d/2+1}).
$$
\endproof

Hence we have, for $0\le h<d$,
\begin{align}\label{pi(A,h)}
|\Pi_d(A,h)|= \frac{q^{h+1}}{d} +  O(q^{\delta(d)}) 
\end{align}
where $\delta(d)= 0$ if $ d=1$, or $\frac{d}2+1$ if $d\ge 2$. 

\begin{remark}\label{keyremark} By
\cite[Corollary 2.6]{BBSR15},
 $$
\pi_q(d;T^{d-h}, B^*) \sim \frac{|\Pi_d|}{\Phi(T^{d-h})} \sim \frac{q^h}{d(1-q^{-1})} \quad \mbox{ as $q\to \infty$}
 $$
holds uniformly for  $3\le h<d$ and all $B^*$ 
relatively prime to $T^{d-h}$. 
When $d\ge 6$, we have $\frac{d}2-1\ge 2$. So if $\frac{d}2-1< h<d$ and $d\geq 6$, then  
$$
|\Pi_d(A,h)|=\sum_{\mu\in \bF_q^\times} \pi_q(d; T^{d-h}, \mu^{-1}B^*) \sim \frac{q^{h+1}}d
$$
and hence $|\Pi_d(A,h)|\gg \frac{q^{h+1}}{d}$. For the remaining cases $d\leq 5$, \eqref{pi(A,h)} gives the same lower bound 
$|\Pi_d(A,h)|\gg \frac{q^{h+1}}{d}$ provided 
$d=1$ or $h>d/2$. Thus,  for $\frac{d}2-1< h<d$,  we can infer 
\begin{align}\label{pilb}
|\Pi_d(A,h)|\gg \frac{q^{h+1}}{d} 
\end{align}
except $(d,h)=(2,1),(3,1),(4,2),(5,2)$ (arising from the cases $\frac{d}2-1<h\le \frac{d}2$ and $2\le d\le 5$).
\end{remark}

Let $B\in \cM_{d-h-1}$ be the unique polynomial such that $A^* \equiv B^* \bmod {T^{d-h}}$.
Consider the abelian extension $E= K_{T^{d-h}}$ of $K= \bF_q(T)$. Take $c\in \mathrm{Gal}(E/K)$ to be the element corresponding to $B^*\in (\bF_q[T]/(T^{d-h}))^\times$ under the isomorphism 
$\mathrm{Gal}(E/K) \cong (\bF_q[T]/(T^{d-h}))^\times$. For any $\mu\in \bF_q^\times$, the set $\Pi_d(E/K,{\mu^{-1}B^*})$ 
can be identified with the set of those $f\in \Pi_d$ such that $f\equiv \mu^{-1}B^*\pmod {T^{d-h}}$. The proof of Lemma~\ref{lm:PNT-short} shows that 
$$
\Pi_d(A,h) = \Big\{\fp\in \Pi_d: \fp(0)^{-1}\fp^* \in \bigcup_{\mu\in \bF_q^\times} \Pi_d (E/K, {\mu^{-1} B^*})\Big\}.
$$
This strongly suggests that Theorem~\ref{thm:STinScKrational} could be used to obtain the distribution of Kloosterman sums in short intervals if one could relate them at places represented by $\fp$ and $\fp^*$. This is indeed our strategy of proof carried out in the next subsection. 

Recall that $a \in\bF_q(T)-\bF_q$. For each irreducible monic polynomial 
$\fp \in \bF_q[T]$ with $\fp\not\in\mathrm{supp}(a)$, the Kloosterman sum at $\fp$ is given by
\begin{align*}
\kl(\fp,a)=\sum_{x\in \bF_{\fp}^\times}\psi(\Tr_{\bF_\fp/\bF_q}(x+\bar a/x)),
\end{align*}
where $\bF_\fp = \bF_q[T]/(\fp(T))$ is the residue field at $\fp$ and $\bar a$ is the image of $a$ in $\bF_\fp$. 
Our key observation is the following.

\begin{lemma}\label{lm:Klidentity}
Let $a(T)\in\bF_q(T)-\bF_q$ and let $\tilde a(T) = a(\frac 1T)$.
For any $\fp \in \bF_q[T]$ with $\fp\not\in\mathrm{supp}(a)$ and $\fp(0) \neq 0$, we have
$$\kl(\fp,a) = \kl(\fp^*,\tilde a).$$
\end{lemma} 

\begin{proof} 
Let $\alpha$ be a root of $\mathfrak p$ in $\overline{\mathbf F}_q$. Then $1/\alpha$ is a root of $\fp^*$. We may identify  both $\bF_{\fp}$ and $\bF_{\fp^*}$ with $\mathbf F_q[\alpha]=\mathbf F_q[1/\alpha]$. We have $\tilde a(1/\alpha)=a(\alpha)$. So 
$$\kl(\fp, a) =\sum_{x\in \mathbf F_q[\alpha]^\times}\psi\Big(\Tr_{\mathbf F_q[\alpha]/\bF_q}\Big(x+\frac{a(\alpha)}{x}\Big)\Big) = \sum_{x\in \mathbf F_q[\alpha]^\times}\psi\Big(\Tr_{\mathbf F_q[\alpha]/\bF_q}\Big(x+\frac{\tilde a(1/\alpha)}{x}\Big)\Big) = \kl(\fp^*,\tilde a).$$
\end{proof}

\begin{remark}
Lemma \ref{lm:Klidentity} admits the following geometric interpretation.
For the automorphism 
$$f: \mathbb G_m\to \mathbb G_m, \quad z\mapsto 1/z,$$ 
we have $f(\fp)=\fp^*$ and $f^*a^*\mathrm{Kl}\cong \tilde a^*\mathrm{Kl}$. So we have 
\begin{eqnarray*}
\mathrm{Kl}({\fp}, a)=\mathrm{Tr}(\mathrm{Frob}_{\fp}, (a^*\mathrm{Kl})_{\bar \fp})=
\mathrm{Tr}(\mathrm{Frob}_{\fp^*}, (\tilde a^*\mathrm{Kl})_{\bar \fp^*})=\mathrm{Kl}({\fp^*}, \tilde a).
\end{eqnarray*}
This indicates that Lemma \ref{lm:Klidentity} applies to any constructible sheaf on $\mathbb P^1$ and is not specific to Kloosterman sums, as pointed out by the referee.
\end{remark}

\subsection{Proof of Theorem \ref{thm:shortinterval}} 

As mentioned before,  if $\fp(0) \ne 0$, then $\fp^*$ is a prime polynomial with the same degree as 
$\fp$ and $\fp^*(0) \ne 0$. So $\fp$ and $\fp^*$ have the same norm. We have equal normalized Kloosterman sums
$$\kl(\fp, a)/\sqrt{N\fp} = \kl(\fp^*, \tilde a)/\sqrt{N\fp^*}= \kl(\fp(0)^{-1}\fp^*, \tilde a)/\sqrt{N(\fp(0)^{-1}\fp^*)}$$
since $\fp^*$ and $\fp(0)^{-1}\fp^*$ are uniformizers of the same place of $K$. 
As noted before, $\fp\in \Pi_d(A,h)$ if and only if $\fp(0)^{-1} \fp^*\in \Pi_d(E/K,{\fp(0)^{-1}B^*})$. 
Furthermore, outside the places $0$ and $\infty$ of $K$, $\fp\not\in\mathrm{supp}(a)$ happens if and only if $\fp(0)^{-1}\fp^*\not\in\mathrm{supp}(\tilde{a})$. 
So we have
\begin{eqnarray*}
&&\Big|\{ \fp \in \Pi_d(A,h):\,\fp\not\in\mathrm{supp}(a),\,\theta_\fp(a)\in I\}
\Big|\\
&=&\Big|\bigcup_{\mu\in\mathbf F_q^\times}\{ \fp \in  \Pi_d(E/K, {\mu B^*} ):\,
\fp \not\in\mathrm{supp}(\tilde{a}), \,\theta_\fp(\tilde a) \in I\} \Big|.
\end{eqnarray*}
Note that
\begin{align*}
&\sum_{\mu\in \bF_q^\times} \sum_{\substack{\fp \in \Pi_d(E/K,{\mu B^*})\\ \fp\not\in\mathrm{supp}
(\tilde a)}}\Sym^k(\theta_\fp({\tilde a}))\\
=\;&\frac{1}{[E:K]}\sum_{\chi \in \widehat{G}} \sum_{\mu \in \bF_q^\times} \overline\chi(\mu B^*)
\sum_{\substack{\fp \in \Pi_d(K)\\ \fp\not\in\mathrm{supp}(\tilde a)\cup\mathrm{supp}\,\mathfrak D_{E/K}}}
\Sym^k( \theta_\fp({\tilde a})) \chi(\mathrm{Frob}_{\fp}),
\end{align*}
where $\widehat{G}$ is the set of all complex characters of $\mathrm{Gal}(E/K)\cong (\bF_q[T]/(T^{d-h}))^\times$. The sum  
$\sum_{\mu \in \bF_q^\times} \overline\chi(\mu B^*)$ vanishes unless $\chi|_{\bF_q^\times} =1$ is trivial, in which case the sum equals $(q-1)\overline{\chi}(B^*)$. Note that the number of $\chi\in   \widehat{G}$ 
satisfying $\chi|_{\bF_q^\times} =1$ is $[E:K]/(q-1)$. Applying Corollary~\ref{twistedsymk}, we find
\begin{align*}
\Big|\sum_{\mu\in \bF_q^\times} \sum_{\substack{\fp \in \Pi_d(E/K,{\mu B^*})\\ \fp\not\in\mathrm{supp}
(\tilde a)}}\Sym^k(\theta_\fp({\tilde a}))\Big|&\leq 6(k+1)\Big((N_{\tilde a}+2) \max(B_{\tilde a}, d-h)+1\Big)\frac{q^{\frac{d}{2}}}{d }\\
&\ll kN_{\tilde a}\max(B_{\tilde a}, d-h) 
\frac{q^{\frac{d}{2}}}{d }.
\end{align*}
Applying Lemma \ref{lm:Niederreiter} to the case  
\begin{eqnarray*}
\Lambda = \Big\{\fp \in \bigcup_{\mu \in \bF_q^\times}\Pi_d(E/K, {\mu B^*} ):\fp\not\in\mathrm{supp}(a)\Big\},\ \ \{\theta_\lambda\}_{\lambda\in \Lambda}=\{\theta_\fp(\tilde a)\}_{\fp \in \Lambda},
\end{eqnarray*}
we obtain 
\begin{eqnarray*}
&&\frac{1}{|\Lambda|}\Big|\bigcup_{\mu\in\mathbf F_q^\times}\{ \fp \in  \Pi_d(E/K, {\mu B^*} ):\,
\fp \not\in\mathrm{supp}(\tilde{a}), \,\theta_\fp(\tilde a) \in I\}\Big |
-\mu_{\mathrm{ST}}(I)\\
&\ll& \frac{1}{|\Lambda|}\sum_{k=1}^M \frac{1}{k}\Big|\sum_{\mu\in \bF_q^\times} \sum_{\substack{\fp \in \Pi_d(E/K,{\mu B^*})\\ \fp\not\in\mathrm{supp}
(\tilde a)}}\Sym^k(\theta_\fp({\tilde a}))\Big|+\frac{1}{M}\\
&\ll& \frac{N_{\tilde a}\max(B_{\tilde a}, d-h)}{|\Lambda|} 
\frac{q^{\frac{d}{2}}}{d }M+\frac{1}{M}.
\end{eqnarray*}
By Remark \ref{keyremark},  we have
$$|\Pi_d(A, h)|=\Big|\bigcup_{\mu \in \bF_q^\times}\Pi_d(E/K, {\mu B^*} )\Big|\gg \frac{q^{h+1}}{d}.$$ So $|\Lambda|\gg \frac{q^{h+1}}{d}$. We thus have
\begin{eqnarray*}
&&\frac{1}{|\Lambda|}\Big|\bigcup_{\mu\in\mathbf F_q^\times}\{ \fp \in  \Pi_d(E/K, {\mu B^*} ):\,
\fp \not\in\mathrm{supp}(\tilde{a}), \,\theta_\fp(\tilde a) \in I\}\Big |
-\mu_{\mathrm{ST}}(I)\\
&\ll& \frac{N_{\tilde a}\max(B_{\tilde a}, d-h)}{|\Lambda|} 
\frac{q^{\frac{d}{2}}}{d }M+\frac{1}{M}\ll 
 {N_{\tilde a}\max(B_{\tilde a}, d-h)} 
q^{\frac{d}{2}-h-1}M+\frac{1}{M}.
\end{eqnarray*}
Taking $M$ to be the integer closest to $({N_{\tilde a}\max(B_{\tilde a}, d-h)} 
q^{\frac{d}{2}-h-1})^{-1/2}$, we get 
$$\frac{1}{|\Lambda|}\Big|\bigcup_{\mu\in\mathbf F_q^\times}\{ \fp \in  \Pi_d(E/K, {\mu B^*} ):\,
\fp \not\in\mathrm{supp}(\tilde{a}), \,\theta_\fp(\tilde a) \in I\}\Big |
-\mu_{\mathrm{ST}}(I)
\ll q^{\frac{d-2h-2}{4}}\sqrt{N_{\tilde a}\max(B_{\tilde a}, d-h)}.$$
Combining with the fact that $N_{\tilde a}= N_a$, we get Theorem  \ref{thm:shortinterval}.

\section{Joint Sato--Tate distribution}

The following lemma facilitates the  study of joint Sato--Tate distributions.

\begin{lemma}\label{lm:NRT-twodim}
Let $n\geqslant1$ be an integer and let $I_i, (1\leqslant i\leqslant n)$ be $n$ given subintervals of $[0,\pi]$. 
Put $I=\prod_{1\leqslant i\leqslant n}I_i$ and $\chi_I(\btheta)=\prod_{1\leqslant i\leqslant n}\chi_{I_i}(\theta_i)$ for $\btheta=(\theta_1,\cdots,\theta_n)\in [0,\pi]^n.$ Then 
for each positive integer $L$, there exist a real valued function $\Delta^\pm(\btheta)$ and
real numbers $a^\pm(\bk)$ for all $\bk=(k_1,\cdots,k_n)\in[0,L]^n,$
such that $$a^\pm(\bk) \ll \prod_{1\leqslant i\leqslant n}\frac{1}{(k_i+1)},\quad \Delta^\pm(\btheta)\ll \frac{1}{L}$$ and
\begin{align}\label{eq:NRT-twodim}
\chi^-(\btheta)\leqslant \chi_I(\btheta)\leqslant \chi^+(\btheta)
\end{align}
with
\begin{align*}
\chi^\pm(\btheta)=\prod_{1\leqslant i\leqslant n}\mu_{\mathrm{ST}}(I_i)+\sum_{\mathbf{0}\neq \bk\in[0,L]^n}a^\pm(\bk)\Sym^{\bk}(\btheta)+\Delta^\pm(\btheta),
\end{align*}
where $\Sym^{\bk}(\btheta)=\prod_{1\leqslant i\leqslant n}\Sym^{k_i}(\theta_i)$ and $\|\chi^\pm\|_\infty\ll 1.$
\end{lemma}

\proof The proof is an iterative application of Lemma \ref{lm:NRT}, and we prove by induction. Note that the case $n=1$ is exactly  Lemma \ref{lm:NRT}.

Suppose that $I$ and $\chi_I$ are of dimension $n$ given as in this lemma. We introduce another interval $J=I_{n+1}\subseteq[0,\pi]$, whose characteristic function $\eta_J$ satisfies 
\begin{align*}
\eta_J^-(\theta)\leqslant \eta_J(\theta)\leqslant \eta_J^+(\theta),
\end{align*}
where $\eta_J^\pm$ are given by
\begin{align*}
\eta_J^\pm(\theta)=\mu_{\mathrm{ST}}(J)+\sum_{1\leqslant k\leqslant L ~}b_k^\pm\Sym^k(\theta)+\delta^\pm(\theta)
\end{align*}
with $\|\eta_J^\pm\|_\infty\ll1$ and $\delta^\pm(\theta)\ll 1/L.$
We now approximate $\chi_I(\btheta)\eta_J(\theta)$ from above and below, to establish the lemma for dimension $n+1$.

Note that $\chi_I\eta_J\leqslant \chi_I^+\eta_J^+.$ We have
\begin{align*}
&\ \ \ \ \chi_I^+(\btheta)\eta_J^+(\theta)\\
&=\eta_J^+(\theta)\Big(\prod_{1\leqslant i\leqslant n}\mu_{\mathrm{ST}}(I_i)+\sum_{\mathbf{0}\neq \bk\in[0,L]^n}a^+(\bk)\Sym^{\bk}(\btheta)\Big)+\eta_J^+(\theta)\Delta^+(\btheta)\\
&=\Big(\mu_{\mathrm{ST}}(J)+\sum_{1\leqslant k\leqslant L ~}b_k^+\Sym^k(\theta)+\delta^+(\theta)\Big)\Big(\prod_{1\leqslant i\leqslant n}\mu_{\mathrm{ST}}(I_i)+\sum_{\mathbf{0}\neq \bk\in[0,L]^n}a^+(\bk)\Sym^{\bk}(\btheta)\Big)\\
&\ \ \ \ +\eta_J^+(\theta)\Delta^+(\btheta)\\
&=\Big(\mu_{\mathrm{ST}}(J)+\sum_{1\leqslant k\leqslant L ~}b_k^+\Sym^k(\theta)\Big)\Big(\prod_{1\leqslant i\leqslant n}\mu_{\mathrm{ST}}(I_i)+\sum_{\mathbf{0}\neq \bk\in[0,L]^n}a^+(\bk)\Sym^{\bk}(\btheta)\Big)+\Delta^+(\btheta,\theta),
\end{align*}
where
\begin{align*}
\Delta^+(\btheta,\theta)
&=\delta^+(\theta)(\chi_I^+(\btheta)-\Delta^+(\btheta))+\eta_J^+(\theta)\Delta^+(\btheta)\ll 1/L.
\end{align*}
Expanding the product, we prove the upper bound part of this lemma for dimension $n+1$.

It remains to deal with the lower bound. To do so, we first observe that $(\chi_I-\chi_I^-)(\eta_J-\eta_J^-)\geqslant0,$ which yields
\begin{align*}
\chi_I\eta_J
&\geqslant \chi_I\eta_J^-+\chi_I^-\eta_J-\chi_I^-\eta_J^-.
\end{align*}
Note that 
\begin{align*}
\chi_I\eta_J^-=\chi_I(\eta_J^- -\eta_J^+)+\chi_I\eta_J^+
\geqslant \chi_I^+(\eta_J^- -\eta_J^+)+\chi_I^-\eta_J^+,
\end{align*}
and similarly,
\begin{align*}
\chi_I^-\eta_J\geqslant \eta_J^+(\chi_I^- -\chi_I^+)+\chi_I^+\eta_J^-.
\end{align*}
Hence
\begin{align*}
\chi_I\eta_J
&\geqslant 2(\chi_I^+\eta_J^-+\chi_I^-\eta_J^+-\chi_I^+\eta_J^+)-\chi_I^-\eta_J^-.
\end{align*}
Following similar arguments to express these $\chi_I^\pm\eta_J^\pm$ in terms of symmetric powers, we complete the proof of the lemma.
\endproof

Lemmata \ref{lm:NRT} and \ref{lm:NRT-twodim} allow us to derive the following $n$-dimensional analogue of 
Niederreiter's inequality (Lemma \ref{lm:Niederreiter}),
which can be regarded as an Erd\H{o}s--Tur\'an inequality for $\mathrm{SU}(2)^{\times n}.$

\begin{lemma}\label{lm:2dNiederreiter}
Given any integer $n\geqslant1$, let $I_1,\cdots, I_n$ be $n$ subintervals of $[0,\pi]$ and $\btheta_\lambda=(\theta_{1,\lambda},\cdots,\theta_{n,\lambda})\in [0,\pi]^n$ for $\lambda\in \Lambda.$ For any positive integer $L\geq 1,$ we have 
\begin{align*}
&\ \ \ \ \frac{|\{\lambda\in \Lambda: \btheta_\lambda\in \prod_{i=1}^n I_i\}|}{|\Lambda|}-\prod_{i=1}^n\mu_{\mathrm{ST}}(I_i)\\
&\ll \frac{1}{|\Lambda|}\mathop{\sum}_{\substack{\mathbf{k}=(k_1,\cdots, k_n)\\ \mathbf{0}\neq \bk\in[0,L]^n}}\frac{1}{\prod_{i=1}^n(k_i+1)}\Big|\sum_{\lambda\in \Lambda} \mathrm{Sym}^{\bk}(\btheta_{\lambda})
\Big|+\frac{1}{L}
\end{align*}
where $\Sym^{\bk}(\btheta)=\prod_{1\leqslant i\leqslant n}\Sym^{k_i}(\theta_i)$ for $\btheta= (\theta_1,\cdots, \theta_n)$.
\end{lemma}

\subsection{Joint Sato--Tate distribution}
For $i=1,\ldots, n$, let $\rho_i:\mathrm{Gal}(\overline K/K)\to\mathrm{GL}(V_i)$ be
$\overline{\mathbf Q}_\ell$-representations of degree $2$ unramified on open subsets $U_i$ of $X$. We assume
$\rho_i$ are punctually $\iota$-pure of weight $0$, and both the arithmetic and the geometric monodromy groups of $\rho_i$ are $\mathrm{SL}(2)$, that is, 
the Zariski closures of images of both $\mathrm{Gal}(\overline K/K)$ and 
$\mathrm{Gal}(\overline K/K\overline{\mathbf F}_q)$ under $\rho_i$ are $\mathrm{SL}(V_i)$. Then for any place $\mathfrak p\in U_i$, 
the eigenvalues of $\iota\rho_i(\mathrm{Frob}_{\mathfrak p})$ are $e^{\theta{i,\mathfrak p}}$ and $e^{-\theta{i,\mathfrak p}}$
for uniquely determined angles $\theta_{i, \mathfrak p}\in [0, \pi]$. Applying Proposition \ref{prop:Frobeinus-keyestimate} 
to the representation $\mathrm{Sym}^{k_1}(\rho_1)\otimes \cdots\otimes\mathrm{Sym}^{k_n}(\rho_n)$, we get the following estimate.

\begin{proposition}\label{prop:estimate_joint}
Let $k_1,\ldots,k_n\geqslant0$ be integers which are not simultaneously zero.
Under the above assumptions, suppose furthermore that 
the representation $\mathrm{Sym}^{k_1}(\rho_1)\otimes \cdots\otimes\mathrm{Sym}^{k_n}(\rho_n)$ has 
neither geometric invariant nor geometric coinvariant.
Let $B$ be the largest break
of the representations $\rho_i|_{I_{\mathfrak p}}$ for all $\mathfrak p\in X-U_1\cap\cdots\cap U_n$ and $i=1,\ldots,n,$ let
$$N=\sum_{\mathfrak p\in X-U_1\cap\cdots\cap U_n} \mathrm{deg} ~\mathfrak p,$$ and let $\mathfrak N=g+1+NB+N$. Then
\begin{align*}
\Bigg|\sum_{\mathfrak p\in U_1\cap\cdots\cap U_n,\,\mathrm{deg}(\mathfrak p) = m} \mathrm{Sym}^{k_1}(\theta_{1, \mathfrak p})
\cdots\mathrm{Sym}^{k_n}(\theta_{n, \mathfrak p}) \Bigg|
&\leq
\frac{q^{m/2}}m (6g+1+N(B+3))(k_1+1)\cdots(k_n+1)\\
&\ll \frac{q^{m/2}}m \mathfrak N(k_1+1)\cdots(k_n+1).
\end{align*}
\end{proposition}

We are now ready to prove a joint Sato--Tate distribution for a finite family of Galois representations with an explicit error term.

\begin{theorem}\label{thm:jointdist}Let $\rho_i:\mathrm{Gal}(\overline K/K)\to\mathrm{GL}(V_i)$ $(i=1,\ldots, n)$ be 
$\overline{\mathbf Q}_\ell$-representations of degree $2$ unramified on $U_i$ punctually $\iota$-pure of weight $0$, and with both the arithmetic and the geometric monodromy groups being $\mathrm{SL}(2)$. Suppose furthermore that for any non-negative integers $k_1,\ldots, k_n$ which are not simultaneously zero,
the representation $\mathrm{Sym}^{k_1}(\rho_1)\otimes \cdots\otimes\mathrm{Sym}^{k_n}(\rho_n)$ has 
neither geometric invariant nor geometric coinvariant. Given subintervals $I_1,\ldots,I_n$ of $[0, \pi]$, let 
\begin{align*}
\Pi_d(U_1\cap \cdots\cap U_n)&=\{\mathfrak p\in U_1\cap\cdots\cap U_n: \,\mathrm{deg}(\mathfrak p)=d\},\\
\Pi_d(\rho_1,\ldots,\rho_n,I_1,\ldots, I_n)&=\{\mathfrak p\in \Pi_d(U_1\cap\cdots\cap U_n):\,
\theta_{1, \mathfrak p}\in I_1, \ldots, \theta_{n, \mathfrak p}\in I_n\}.
\end{align*}
Then we have 
\begin{align*}
\frac{|\Pi_d(\rho_1,\ldots, \rho_n,I_1, \ldots, I_n)|}{|\Pi_d(U_1\cap\cdots\cap U_n)|}=\mu_{\mathrm{ST}}(I_1)\cdots\mu_{\mathrm{ST}}(I_n)
+O(q^{-\frac{d}{2(n+1)}}\mathfrak N^{\frac{1}{n+1}}),
\end{align*}
where the implied $O$-constant is absolute and $\mathfrak N$ is as in Proposition $\ref{prop:estimate_joint}.$
\end{theorem}

\begin{proof} By Lemma \ref{lm:2dNiederreiter}, for any positive integer $L$, we have 
\begin{align*}
&\ \ \ \ \frac{|\Pi_d(\rho_1,\ldots, \rho_n,I_1, \ldots,I_n)|}{|\Pi_d(U_1\cap\cdots\cap U_n)|}-\mu_{\mathrm{ST}}(I_1)\cdots\mu_{\mathrm{ST}}(I_n)\\
&\ll\frac{1}{|\Pi_d(U_1\cap\cdots\cap U_n)|} \mathop{\sum\sum}_{\substack{0\leq k_1,\ldots, k_n\le L\\(k_1, 
\ldots,k_n)\not=(0,\ldots,0)}}
\frac{1}{(k_1+1)\cdots(k_n+1)}\\
&\qquad\qquad\qquad\qquad\qquad\qquad 
\times \Big|\sum_{\mathfrak p \in U_1\cap \cdots\cap U_n,\,\mathrm{deg}(\mathfrak p)=d} 
\mathrm{Sym}^{k_1}(\theta_{1,\mathfrak p})\cdots
\mathrm{Sym}^{k_n}(\theta_{n,\mathfrak p})\Big|+ \frac{1}{L}.
\end{align*} 
Combining with Proposition \ref{prop:estimate_joint}, we get
\begin{align*}
\frac{|\Pi_d(\rho_1,\ldots, \rho_n,I_1,\ldots, I_n)|}{|\Pi_d(U_1\cap\cdots\cap U_n)|}-\mu_{\mathrm{ST}}(I_1)\cdots\mu_{\mathrm{ST}}(I_n)&\ll
\frac{1}{|\Pi_d(U_1\cap\cdots\cap U_n)|} \frac{q^{d/2}}{d}\mathfrak N L^n
+\frac{1}{L}.
\end{align*}
By Lemma \ref{lm:PNT}, for large $d$, we have 
\begin{align*}
|\Pi_d(U_1\cap \cdots\cap U_d)|&=\frac{q^d}{d}\Big(1+O(q^{-d/2})\Big)\gg\frac{q^d}{d}.
\end{align*}
We complete the proof of Theorem \ref{thm:jointdist} by taking $L$ to be the integer closest to $(q^{-\frac{d}{2}} \mathfrak N)^{-\frac{1}{n+1}}$.
\end{proof}

\subsection{Proof of Theorem \ref{thm:jointdistKl}} 
Denote by $\rho_i$ the representations of $\mathrm{Gal}(\overline K/K)$ associated to the sheaves $a_i^*\mathrm{Kl}(1/2)$ for $i=1,\ldots, n$.
We apply Theorem \ref{thm:jointdist}. In order to do this, we need to verify that for any nonnegative integers $k_1,\ldots, k_n$ not simultaneously zero, $\mathrm{Sym}^{k_1}(\rho_1)\otimes\cdots \mathrm{Sym}^{k_n}(\rho_n)$ has
neither geometric invariant nor geometric coinvariant. Let $k_{i_1},\ldots, k_{i_m}$ be those strictly positive integers among $k_1,\ldots, k_n$.
By \cite[4.1.3 and 4.1.4]{Ka88}, $\rho_{i_1}$ is geometrically self dual, that is, as a representation of $G:=\mathrm{Gal}(\overline K/K\overline{\mathbf F}_q)$,  $\rho_{i_1}$ is isomorphic to its contragredient representation $\rho_{i_1}^\vee$. So we have 
\begin{align*}
\Big(\mathrm{Sym}^{k_{i_1}}(\rho_{i_1})\otimes \cdots\otimes\mathrm{Sym}^{k_{i_m}}(\rho_{i_m})\Big)^G
&\cong \Big(\mathrm{Sym}^{k_{i_1}}(\rho_{i_1})^\vee\otimes \mathrm{Sym}^{k_{i_2}}(\rho_{i_2})\otimes \cdots\otimes\mathrm{Sym}^{k_{i_m}}(\rho_{i_m})\Big)^G\\
&\cong \mathrm{Hom}_G\Big(\mathrm{Sym}^{k_{i_1}}(\rho_{i_1}), \mathrm{Sym}^{k_{i_2}}(\rho_{i_2})\otimes \cdots\otimes\mathrm{Sym}^{k_{i_m}}(\rho_{i_m})\Big).
\end{align*}
By \cite[Sommes trig.7.8 (iii)] {De77}, the Kloosterman sheaf $\mathrm{Kl}$ is tamely ramified at $0$ with a unipotent monodromy of a single Jordan block. Outside $0$, the Kloosterman sheaf is either unramified (at finite places) or wildly ramified (at $\infty$). Choose a place $\mathfrak p$ of $K$ which is a zero of $a_{i_1}$ but not a zero of $a_{i_2},\ldots, a_{i_m}$. Then $\rho_{i_1}$ has nontrivial tame ramification at $\mathfrak p$, but $\rho_{i_2},\ldots, \rho_{i_m}$ are either unramified or wildly ramified at $\mathfrak p$. So 
$\mathrm{Sym}^{k_{i_1}}(\rho_{i_1})$ and $\mathrm{Sym}^{k_{i_2}}(\rho_{i_2})\otimes \cdots\otimes\mathrm{Sym}^{k_{i_m}}(\rho_{i_m})$ are not isomorphic as representations of $\mathrm{Gal}(\overline K/K\overline{\mathbf F}_q)$. By Schur's lemma, we have 
$$\mathrm{Hom}_G\Big(\mathrm{Sym}^{k_{i_1}}(\rho_{i_1}), \mathrm{Sym}^{k_{i_2}}(\rho_{i_2})\otimes \cdots\otimes\mathrm{Sym}^{k_{i_m}}(\rho_{i_m})\Big)=0.$$ So $\mathrm{Sym}^{k_{i_1}}(\rho_{k_{i_1}})\otimes\cdots \otimes\mathrm{Sym}^{k_{i_m}}(\rho_{k_{i_m}})$ has no geometric invariant. Since this representation of $G$ is semisimple, it has no geometric coninvariant.
This completes the proof of Theorem \ref{thm:jointdistKl}.

\appendix

\section{An effective Chebotarev density theorem}

A consequence of the Riemann Hypothesis for curves over finite fields is an effective version of the Chebotarev density theorem for function fields, which can be regarded as an analogue of the work of Lagarias and Odlyzko \cite{LO77} for number fields. This was identified in \cite{MS94} and we provide our self-contained exposition below.

\begin{theorem}[Effective Chebotarev density theorem]\label{thm:effChebo}
Let $K$ be a function field of genus $g$. Let $E/K$ be a finite Galois extension such that $E$ has the same constant field $\mathbf F_q$ as $K$, $\mathfrak D_{E/K}$ its discriminant,  
$C$ a conjugacy class in $\mathrm{Gal}(E/K)$, $\Pi_m(K)$ the set of places $\mathfrak p$ of $K$ 
with $\mathrm{deg}(\mathfrak p)=m$, $\Pi_m(E/K)$ the subset of all places $\mathfrak p$ in $\Pi_m(K)$
unramified in $E$ , and 
\begin{align*}
\Pi_m(E/K,C)
:= \{\mathfrak p\in \Pi_m(E/K): \mathrm{deg}(\mathfrak p)=m, \; \mathrm{Frob}_\mathfrak p=C\}.
\end{align*}
We have

(i)
\begin{align*}
\Big| |\Pi_m(E/K,C)|&-\frac{|C|}{[E:K]}|\Pi_m(E/K)|\Big|
\leqslant |C|\frac{q^{m/2}}{m}\Big(6g+1+3\deg(\fD_{E/K})\Big),
\end{align*} 

(ii) 
\begin{align*}
\Big| |\Pi_m(E/K,C)|&-\frac{|C|}{[E:K]}|\Pi_m(K)|\Big|
\leqslant |C|\frac{q^{m/2}}{m}\Big(6g+1+4\deg(\fD_{E/K})\Big),
\end{align*}

(iii) 
\begin{align*}
\Big||\Pi_m(E/K,C)|-\frac{|C|}{[E:K]}\frac{q^m}{m}\Big|\leq  |C|\frac{q^{m/2}}{m}\Big(12g+5+4\deg(\fD_{E/K})\Big).
\end{align*}
\end{theorem}

Write $G$ for $\mathrm{Gal}(E/K)$. Since $G$ is a finite group, we may identify $\ell$-adic representations with complex representations of $G$. Denote by $\widehat G$ the set of irreducible representations of $G$. For $\rho \in \widehat G$, let $\chi_\rho = \mathrm{Tr}(\rho)$. 
Consider the character sum
$$T_\rho(m)
:=\sum_{\fp\not\in\mathrm{supp}\,\mathfrak D_{E/K},\,\mathrm{deg}(\mathfrak p) = m}  \chi_\rho(\mathrm{Frob}_{\fp}).$$
The orthogonality of $\chi_\rho$ for $\rho \in \widehat G$ allows us to express $|\Pi_m(E/K,C)|$ as a linear combination of $T_\rho(m).$

\begin{lemma}\label{average} 
For any conjugacy class $C$ in $G$, we have 
\begin{align*}
|\Pi_m(E/K,C)|&=\frac{|C|}{[E:K]} \sum_{\rho \in \widehat{G}}\bar\chi_\rho(C)T_\rho(m).
\end{align*}
\end{lemma}

\begin{lemma}\label{Hurwitz}
We have 
$$\sum_{\rho \in \hat G}\mathrm{dim}(\rho)^2=[E:K],\quad
\sum_{\rho\not=1,\rho \in \widehat G}\mathrm{dim}(\rho)a(\rho)=\mathrm{deg}(\mathfrak D_{E/K}),$$
where $a(\rho)$ is as in \eqref{eq:Artinconductor}.
\end{lemma}

\begin{proof} Let $X'$ be the normalization in $E$ of the smooth projective curve $X$ with the function field $K$. 
Then the $\zeta$-function of $X'$ is equal to
$$\prod_{\rho \in \widehat G}L(\rho, t)^{\dim(\rho)}$$ 
since it is equal to the $L$-function of the regular representation of 
$G$, and each irreducible representation $\rho$ of  $G$ appears in the regular 
representation with multiplicity $\mathrm{dim}(\rho)$. The degree of the $\zeta$-function of $X'$ is 
$2g'-2$, where $g'$ is the genus of $X'$. 
Counting the degree of $\prod_{\rho}L(\rho, t)^{\mathrm{dim}(\rho)}$, we get
$$2g'-2=\sum_\rho((2g-2)\mathrm{dim}(\rho)+a(\rho))\mathrm{dim}(\rho)=(2g-2)\sum_\rho\mathrm{dim}(\rho)^2+
\sum_{\rho\not=1}a(\rho)\mathrm{dim}(\rho).$$ Here we use the fact that $a(\rho)=0$ for the trivial representation. 
Counting the dimension of the regular representation decomposed as the direct sum of irreducible representations $\rho$ 
with multiplicities $\mathrm{dim}(\rho)$, we get
$$\sum_\rho\mathrm{dim}(\rho)^2=[E:K].$$ 
By the Hurwitz formula, we have 
$$2g'-2=(2g-2)[E:K]+\mathrm{deg}(\mathfrak D_{E/K}),$$ 
from which our assertion follows. 
\end{proof} 

\subsection*{Proof of Theorem \ref{thm:effChebo}}
By Lemma \ref{average}, we have
\begin{eqnarray}\label{delta}
|\Pi_m(E/K,C)|-\frac{|C|}{[E:K]}|\Pi_m(E/K)|=\frac{|C|}{[E:K]} \sum_{\rho\not=1, \rho\in \hat G }\bar \chi_\rho(C)T_\rho(m).
\end{eqnarray} 
Any nontrivial irreducible representation $\rho$ has neither geometric invariant
nor geometric coinvariant. Moreover, $\rho$ is punctually $\iota$-pure of weight $0$, and unramified on $X-\mathrm{supp}\,\mathfrak D_{E/K}$.
From Proposition \ref{prop:Frobeinus-keyestimate}, it follows that
\begin{align}\label{T}
|T_\rho(m) |\leq
\frac{q^{m/2}}m \Big((6g+1)\dim(\rho)+a(\rho)\Big)+\omega_{E/K}\dim(\rho),
\end{align} where $\omega_{E/K}$ is the number of distinct places in $\mathrm{supp}\,\mathfrak D_{E/K}$. Note that $\omega_{E/K}\le 
\mathrm{deg}(\mathfrak D_{E/K})$.
From the equation (\ref{delta}), the inequality (\ref{T}) and Lemma \ref{Hurwitz}, we get
\begin{eqnarray*}
&& ||\Pi_m(E/K,C)|-\frac{|C|}{[E:K]}|\Pi_m(E/K)||\\
&\leqslant& \frac{|C|}{[E:K]}  \sum_{\rho\not=1}\dim(\rho)\frac{q^{m/2}}m\Big((6g+1)\dim(\rho)+a(\rho)\Big)
+\frac{|C|}{[E:K]}\omega_{E/K}\sum_{\rho\not=1}\dim(\rho)^2\\
&=&\frac{|C|}{[E:K]}  \frac{q^{m/2}}m\Big((6g+1)([E:K]-1)+\mathrm{deg}(\mathfrak D_{E/K})\Big)+\frac{|C|}{[E:K]}\omega_{E/K}([E:K]-1)\\
&\le& |C| \frac{q^{m/2}}m \Big(6g+1+3\mathrm{deg}(\mathfrak D_{E/K})\Big).
\end{eqnarray*}
This proves (i).
(ii) follows from (i) and the fact that 
\begin{align*}
0\leq |\Pi_m(K)|-\Pi_m(E/K)|\leq\frac{\deg(\fD_{E/K})}{m}.
\end{align*}
(iii) follows from (ii) and Lemma \ref{lm:PNT}.

\bibliographystyle{plainnat}

\end{document}